\documentclass[12pt]{amsart}
\usepackage{amscd,amsmath,amsthm,amssymb}
\usepackage{amsfonts,amssymb,amscd,amsmath,enumerate,verbatim}
\usepackage{tikz, float} \usetikzlibrary {positioning}
\usepackage[left]{lineno}
\usepackage{pstcol,pst-plot,pst-3d}
\usepackage{color}
\usepackage{pstricks}
\usepackage{stmaryrd}
\usepackage[utf8]{inputenc}
\usepackage{pstricks-add}
\usepackage{graphicx}
\newpsstyle{fatline}{linewidth=1.5pt}
\newpsstyle{fyp}{fillstyle=solid,fillcolor=verylight}
\definecolor{verylight}{gray}{0.97}
\definecolor{light}{gray}{0.9}
\definecolor{medium}{gray}{0.85}
\definecolor{dark}{gray}{0.6}

 \def\NZQ{\mathbb}               

 \def\ZZ{{\NZQ Z}}

 \def\frk{\mathfrak}               

 \def\mm{{\frk m}}

 \def\Jc{{\mathcal J}}
 \def\Rc{{\mathcal R}}
 
 \def\G{{\mathcal G}}

 \def\P{{\mathcal P}}

\def\Cc{{\mathcal C}}

 \def\xb{{\mathbf x}}

 \def\fb{{\mathbf f}}
 \def\opn#1#2{\def#1{\operatorname{#2}}} 
 \opn\chara{char} \opn\length{\ell} \opn\pd{pd} \opn\rk{rk}
 \opn\projdim{proj\,dim} \opn\injdim{inj\,dim} \opn\rank{rank}
 \opn\depth{depth} \opn\grade{grade} \opn\height{height}
 \opn\embdim{emb\,dim} \opn\codim{codim}
 
 \opn\Tr{Tr} \opn\bigrank{big\,rank}
 \opn\superheight{superheight}\opn\lcm{lcm}
 \opn\trdeg{tr\,deg}
 \opn\reg{reg} \opn\lreg{lreg} \opn\ini{in} \opn\lpd{lpd}
 \opn\size{size} \opn\sdepth{sdepth}
 \opn\link{link}\opn\fdepth{fdepth}\opn\lex{lex}
 \opn\tr{tr}
 \opn\type{type}
 \opn\gap{gap}
 \opn\arithdeg{arith-deg}
 \opn\revlex{revlex}
 \opn\cut{cut}
 \opn\div{div} \opn\Div{Div} \opn\cl{cl} \opn\Cl{Cl}
 \opn\Spec{Spec} \opn\Supp{Supp} \opn\supp{supp} \opn\Sing{Sing}
 \opn\Ass{Ass} \opn\Min{Min}\opn\Mon{Mon}
 \opn\Ann{Ann} \opn\Rad{Rad} \opn\Soc{Soc}
 \opn\Im{Im} \opn\Ker{Ker} \opn\Coker{Coker} \opn\Am{Am}
 \opn\Hom{Hom} \opn\Tor{Tor} \opn\Ext{Ext} \opn\End{End}
 \opn\Aut{Aut} \opn\id{id}
 
 \opn\nat{nat}
 \opn\pff{pf}
 \opn\Pf{Pf} \opn\GL{GL} \opn\SL{SL} \opn\mod{mod} \opn\ord{ord}
 \opn\Gin{Gin} \opn\Hilb{Hilb}\opn\sort{sort}
 \opn\PF{PF}\opn\Ap{Ap}
 \opn\mult{mult}
 \opn\bight{bight}
 \opn\aff{aff}
 \opn\relint{relint} \opn\st{st}
 \opn\lk{lk} \opn\cn{cn} \opn\core{core} \opn\vol{vol}  \opn\inp{inp} \opn\nilpot{nilpot}
 \opn\link{link} \opn\star{star}\opn\lex{lex}\opn\set{set}
 \opn\width{wd}
 \opn\Fr{F}
 \opn\QF{QF}
 \opn\G{G}
 \opn\type{type}\opn\res{res}
 \opn\conv{conv}
 \opn\Deg{Deg}
 \opn\Sym{Sym}
 \opn\gr{gr}
 
 \def\pot#1#2{#1[\kern-0.28ex[#2]\kern-0.28ex]}

 \opn\dirlim{\underrightarrow{\lim}}
 \opn\inivlim{\underleftarrow{\lim}}
 \let\union=\cup
 \let\sect=\cap

 \let\iso=\cong
 \let\Union=\bigcup
 \let\Sect=\bigcap

 \let\to=\rightarrow
 
 \def\Implies{\ifmmode\Longrightarrow \else
         \unskip${}\Longrightarrow{}$\ignorespaces\fi}
 \def\implies{\ifmmode\Rightarrow \else
         \unskip${}\Rightarrow{}$\ignorespaces\fi}
 \def\iff{\ifmmode\Longleftrightarrow \else
         \unskip${}\Longleftrightarrow{}$\ignorespaces\fi}

 \let\:=\colon
 \newtheorem{Theorem}{Theorem}[section]
 \newtheorem{Lemma}[Theorem]{Lemma}
 \newtheorem{Corollary}[Theorem]{Corollary}
 \newtheorem{Proposition}[Theorem]{Proposition}
 \newtheorem{Remark}[Theorem]{Remark}
 
 \newtheorem{Example}[Theorem]{Example}
 
 \newtheorem{Definition}[Theorem]{Definition}

 \let\epsilon\varepsilon
 \let\kappa=\varkappa
 \def\qed{\ifhmode\textqed\fi
       \ifmmode\ifinner\quad\qedsymbol\else\dispqed\fi\fi}
 \def\textqed{\unskip\nobreak\penalty50
        \hskip2em\hbox{}\nobreak\hfil\qedsymbol
        \parfillskip=0pt \finalhyphendemerits=0}
 \def\dispqed{\rlap{\qquad\qedsymbol}}
 
 \opn\dis{dis}
 \def\pnt{{\raise0.5mm\hbox{\large\bf.}}}
 
 \opn\Lex{Lex}

\begin{document}

 \title{Graded ideals of K\"onig type}

\author {J\"urgen Herzog, Takayuki Hibi and  Somayeh Moradi }

\address{J\"urgen Herzog, Fachbereich Mathematik, Universit\"at Duisburg-Essen, Campus Essen, 45117
Essen, Germany} \email{juergen.herzog@uni-essen.de}

\address{Takayuki Hibi, Department of Pure and Applied Mathematics,
Graduate School of Information Science and Technology, Osaka
University, Suita, Osaka 565-0871, Japan}
\email{hibi@math.sci.osaka-u.ac.jp}

\address{Somayeh Moradi, Department of Mathematics, School of Science, Ilam University,
P.O.Box 69315-516, Ilam, Iran}
\email{so.moradi@ilam.ac.ir}

\dedicatory{ }
\thanks{The second author was supported by JSPS KAKENHI 19H00637.}

\begin{abstract}
Inspired by the notion of K\"onig graphs we introduce graded ideals of K\"onig type with respect to a monomial order $<$. It is shown that if $I$ is of K\"onig type, then the Cohen--Macaulay property of $\ini_<(I)$ does not depend on the characteristic of the base field. This happens to be the case also for $I$ itself when $I$ is a binomial edge ideal. Attached to an ideal of K\"onig type is a sequence of linear forms, whose elements are variables or differences of variables. This sequence is a system of parameters for $\ini_<(I)$,  and is a potential system of parameters for $I$ itself. We study in detail the ideals of K\"onig type among the   edge ideals, binomial edge ideals and the toric ideal of a Hibi ring and use the K\"onig property to determine explicitly their canonical module.
\end{abstract}

\maketitle

\setcounter{tocdepth}{1}

\section*{Introduction}

For a graded ideal $I$ of a polynomial ring $S=K[x_1,\ldots,x_n]$ over a field $K$, the Cohen-Macaulay property of the ring $S/I$ may depend on the characteristic of the base field $K$ in general.
Even when we restrict to the case that $I$ is the edge ideal of a graph $G$, this dependency may occur. Such an example is given in \cite{Katzman}. Nevertheless, for some families of graphs it is shown that the Cohen-Macaulayness of $S/I(G)$ is independent of $K$.
When $G$ is a bipartite graph, the first and second authors of this paper in \cite{HH1} presented a combinatorial condition on the graph $G$ equivalent to $S/I(G)$ be Cohen-Macaulay. In \cite{CRT} characterizing Cohen-Macaulay graphs was extended to very well-covered graphs which implies that the Cohen-Macaulay property of $S/I(G)$ is independent of $K$, when $G$ is a very well-covered graph (see also \cite{CCR}). If $G$ is a graph with no isolated vertex, $G$ is very well-covered if and only if it is an unmixed K\"onig graph. Recall that a graph $G$ is called K\"onig, if its matching number coincides with its vertex cover number.
Motivated by this, we look in a more general frame for conditions on a graded ideal $I\subset S$ for which the Cohen-Macaulay property of the ring $S/I$ is independent of $K$ and to this aim we introduce graded ideals of K\"onig type which generalize edge ideals of K\"onig graphs.

 In \cite{HM}, the first and third authors studied monomial ideals of K\"onig type and characterized them in terms of the existence of systems of parameters for $S/I$ which consist of elements of the form $x_i-x_j$.
Bringing to mind the characteristic independency of K\"onig graphs together with this result, this natural question arises: for a graded ideal $I$, does the existence of systems of parameters in a particular form for $S/I$ imply the characteristic independence of Cohen-Macaulayness?
We investigate this question when $S/I$ admits a special system of parameters. By a special system of parameters we mean a system of parameters $s_1,\ldots,s_d$, where each $s_j$ is either of the form $x_i$ or of the form $x_i-x_j$.
Indeed we study binomial ideals whose quotient rings admit special systems of parameters. Here we call an ideal $I$ a binomial ideal if the generators of $I$ are  monomials or binomials of the form $u-v$ with $u$ and $v$ monomials. This includes the case that all generators are monomials or all generators are binomials.
It turns out that for a binomial ideal $I$, the existence of a special system of parameters for $S/I$ gives a positive answer to the above question. For binomial edge ideals of K\"onig type as well as monomial ideals of K\"onig type, such special systems of parameters always exist.
We study further monomial ideals and binomial edge ideals of K\"onig type and investigate when the defining ideal of a Hibi ring is of K\"onig type.

The paper proceeds as follows. In the first section we define a graded ideal $I$ to be of K\"onig type with respect to a homogeneous sequence $f_1,\ldots,f_h$ coming from a minimal generating set of $I$  and an order $<$
if $h=\height I$ and $\ini_<(f_1),\ldots,\ini_<(f_h)$ is a regular sequence. Attached to this sequence and this order we consider a natural sequence of linear forms of length $d=\dim S/I$ which has the potential to be a special system of parameters for $S/I$.
In Theorem~\ref{independent}, we show that for a binomial ideal $I$ if $S/I$ admits a special system of parameters, the Cohen-Macaulay property of $S/I$ does not depend on the base field. Using this theorem and the characterization of monomial ideals of K\"onig type, we deduce that
for any graded ideal of K\"onig type with respect to $<$, the Cohen--Macaulayness of $\ini_<(I)$ does not depend on the base field (see Corollary~\ref{initial}).

In Section~\ref{two}, we concentrate on monomial ideals of K\"onig type.  In Theorem~\ref{IB},  a combinatorial condition is given which is equivalent to the Cohen-Macaulayness of $S/I$, when $I$ is a monomial ideal of  K\"onig type. This generalizes \cite[Proposition 3.2]{CRT}.
In Theorem~\ref{canonical2}, for a  Cohen-Macaulay K\"onig graph,  combinatorial descriptions for the canonical module and the Cohen-Macaulay type of $S/I(G)$ are presented.
Next in Theorem~\ref{characterizekoenig}, for a K\"onig graph $G$ we show that $S/I(G)$ is Cohen--Macaulay if and only if $\alpha(G)=|\Delta_{G_0}|$, where $\alpha(G)$ is the number of vertex covers of $G$ of cardinality $\tau(G)$, $G_0$ is some graph attached to $G$ and $\Delta_{G_0}$ is the independence complex of $G_0$.
When $I$ is the edge ideal of an unmixed graph $G$ without isolated vertices, it is known that $I$ is of K\"onig type if and only if $G$ is a very well-covered graph (see~\cite[Lemma 17]{CCR}). In Theorem~\ref{very}, we consider an unmixed monomial ideal of K\"onig type $I$ and show that the polarization $I^\wp$ of $I$ is very well-covered. An example follows which shows that the converse may not hold in general.

In Section~\ref{three}, we study binomial edge ideals of K\"onig type. We call a subgraph $P$ of $G$ a semi-path if each component of $P$ is a path graph.
In Theorem~\ref{lenghtsemipaths} it is shown that for a graph $G$ on $[n]$ the binomial edge ideal $J_G$ is of  K\"onig type if and only if there exists a semi-path in $G$ of length $2n-d$, where $d=\dim S/J_G$. There is a close relation between binomial edge ideals of K\"onig type and traceable graphs , i.e., graphs which contain a path meeting all vertices of the graph. Indeed if $G$ is traceable,  then $J_G$ is of K\"onig type.  Also if $J_G$ is of K\"onig type and unmixed, then $G$ is traceable (see Proposition~\ref{somayehchanged}).  Giving a special system of parameters for a binomial edge ideal of K\"onig type in Theorem~\ref{parametersystem}, we recover a result
in  ~\cite{BMS} which indicates that for a traceable graph, the Cohen-Macaulay property of $S/J_G$ does not depend on the base field. We close this section by a proposition which gives a description for the canonical module of  $S/J_G$, when $G$ is traceable and $J_G$ is Cohen-Macaulay.

Finally in Section~\ref{four}, we study when the toric ideal of a Hibi ring $\Rc_K[L]$ is of K\"onig type with respect to $f_{i, j}$ and $<_{\rm rev}$, where $f_{i, j} = x_ix_j - (x_i \wedge x_j)(x_i \vee x_j)$ for which $x_i$ and $x_j$ are incomparable in $L$.
We show in Theorem~\ref{binomialhibi} that  $I_L$ is of K\"onig type with respect to $f_{i, j}$ and $<_{\rm rev}$ if and only if
$L$ is thin and this is the case when the incomparability graph ${\rm incom(L)}$ of $L$ is bipartite.

\section{Graded ideals of K\"onig type and parameter systems}
\label{one}

We recall some concepts about graphs.
Let $G$ be a finite simple graph with the vertex set $V(G)$ and the edge set $E(G)$.
A {\em matching} of $G$ is a subset $\{e_1,\ldots,e_m\}$ of $E(G)$ such that $e_i$'s are pairwise disjoint.
The maximum cardinality of matchings of $G$ is called the {\em matching number} of $G$ and is denoted by $m(G)$. A {\em vertex cover} of $G$ is a subset $C$ of $V(G)$ which has at least one element from each edge of $G$. A vertex cover $C$ is called a {\em minimal vertex cover} if no proper subset of $C$ is a vertex cover.
The minimum cardinality of vertex covers of $G$ is called the {\em vertex cover number} of $G$ and is denoted by $\tau(G)$.

Let $G$ be a graph on $[n]$. The edge ideal $I(G)$  of $G$ is the monomial ideal  whose generators are the monomials $x_ix_j$ for which $\{i,j\}\in E(G)$. It is easy to see that the  height of $I(G)$  is equal to $\tau(G)$  and $m(G)$ is the maximal length of a regular sequence  among the generators of $I(G)$. It follows that $m(G)\leq \tau(G)$ since $m(G)\leq \grade(I(G))=\height I(G)=\tau(G)$.  In 1931,  K\"onig showed that $m(G)= \tau(G)$, if $G$ is bipartite.  Now, any graph $G$, satisfying $m(G)= \tau(G)$ is called a {\em K\"onig graph}.

\medskip
This concept,  as defined for graphs,  leads us to the following

\begin{Definition}
{\em Let $K$ be field, $S=K[x_1,\ldots,x_n]$ be the polynomial ring in $n$ variables over $K$,  and let $I\subset S$ be a graded ideal of height $h$. We say that $I$ is of {\em K\"onig type},  if there exists a sequence $\fb=f_1,\ldots,f_h$ of homogeneous polynomials which forms part of a minimal system of generators of $I$ and a monomial order $<$ on $S$ such that   $\ini_<(f_1),\ldots,\ini_<(f_h)$ is a regular sequence.
}
\end{Definition}

If we want to be more specific we say that $I$ is of K\"onig type with respect to the sequence $\fb$ and the monomial order $<$.

\begin{Remark}
{\em If $K$ is infinite, one can always find  a homogeneous regular sequence $f_1,\ldots , f_h\in I$. The condition that $\ini_<(f_1),\ldots,\ini_<(f_h)$ be a regular sequence for some monomial order and part of a minimal system of generators is much stronger. This is why not all graded ideals are of K\"onig type.}
\end{Remark}

\begin{Example}
{\em Consider the ideal $I=(f_1,f_2)$ with $f_1=x_1x_2-x_2x_3$ and  $f_2=x_1x_3-x_3^2$. Then $I$ is of height 2. The only two monomials in the support of $f_1$ and $f_2$   which form a regular sequence are the monomials $x_1x_2$ and $x_3^2$.  But there is no monomial order $<$ with $\{\ini_<(f_1), \ini_<(f_2)\}= \{x_1x_2, x_3^2\}$. Thus   $I$ is not of K\"onig type. }
\end{Example}

Suppose now that $I$ is of K\"onig type  with respect to $\fb=f_1,\ldots,f_h$ and  $<$, and  let $u_i=\ini_<(f_i)$ for all $i$.
Then we define  a sequence  $C$  of linear forms attached to $\fb$ and $<$ as follows:
let $$A=\{x_i\:\; x_i\nmid u_j \text{ for all $j$}\},$$ and for each $j$ let $$B_j=\{x_i\:\; x_i\mid u_j\}.$$
 Since $u_1,\ldots,u_h$ is a regular sequence, the sets $B_j$ are pairwise disjoint and of course $A\sect B_j=\emptyset$ for all $j$. Therefore,  $n=|A|+\sum_{j=1}^h|B_j|=n$.

Finally for each $j$  let $i_j$ be the smallest integer with $x_{i_j}\in B_j$ and let $$C_j=\{x_k-x_{i_j}\:\; x_k\in B_j, k\neq i_j\}.$$  Then we set
\begin{eqnarray}
\label{goodsequence}
C=A\union C_1\union \cdots \union C_h.
\end{eqnarray}

Note that $|C|=n-h=\dim S/I$, because   $|C_j|=|B_j|-1$ for $j=1,\ldots,h$. Hence $C$ is a potential system of parameters of $S/I$ attached to $f_1,\ldots,f_h$ and $<$.

\begin{Example}{\em
Let $R=K[x_1,x_2]/I$, where $I=(x_1x_2-x_2^2)$. Then $I$ is of  K\"onig type with respect to $f= x_1x_2-x_2^2$ for any monomial order. If $x_1>x_2$, then the
attached sequence is $x_2-x_1$ and $R/(x_2-x_1)R\iso K[x_1]$. On the other hand,   if $x_2>x_1$, then the attached sequence is $x_1$, and $R/(x_1)R\iso K[x_2]/(x_2^2)$. Hence in the second case, the attached sequence (\ref{goodsequence}) is a system of parameters of $R$, while  in the first case it is not.}
\end{Example}

Let $R$ be a standard graded $K$-algebra of dimension $d$. Having in mind attached sequences of linear forms attached to  ideals of  K\"onig type, we call a  system of parameters $s_1,\ldots,s_d$ of $R$  {\em special} with respect to a given $K$-basis $x_1,\ldots,x_n$ of $R_1$, if each $s_j$ is either of the form $x_i$ or of the form $x_k-x_i$.

For binomial ideals,  special systems of parameters have nice properties. Here we  call an ideal $I$ a {\em binomial ideal} if the generators of $I$ are  monomials and binomials of the form $u-v$ with $u$ and $v$ monomials. This includes the case that all generators are monomials or all generators are binomials.

\begin{Theorem}
\label{independent}
Let $I\subset S$ be a binomial ideal and suppose that $R=S/I$ admits  a special system of parameters. Then the Cohen-Macaulay property of $R$ does not depend on the base field.
\end{Theorem}

\begin{proof}
Let $L\subset S$  be the ideal generated by the special system of parameters. Then $S/L\iso T$ where $T$ is a polynomial ring over $K$ in $h$ variables, and $R/LR$  can be written as $T/J$,  where $J$ is again a binomial ideal because the generators of $J$ are binomials which are obtained from the binomials generating $I$ by replacing some variables by $0$ or identifying different variables.

Denote by $e(A)$ the multiplicity of a standard graded $K$-algebra $A$. Then  it is known (see for example \cite[Proposition 1.2]{HM}) that  $R$ is Cohen-Macaulay if and only if $e(R)=e(R/LR)$ (in which case the system of parameters forms a regular sequence). The multiplicities $e(R)$ and $e(R/LR)$ are determined by their corresponding Hilbert series. Thus  Lemma~\ref{takayuki} completes the proof.
\end{proof}

\begin{Lemma}
\label{takayuki}
Let $I\subset S$ be a binomial ideal. Then the Hilbert series of $S/I$ does not depend on the base field.
\end{Lemma}

\begin{proof}
First of all, it is known \cite[Corollary 6.1.5]{HH} that, for an arbitrary monomial order $<$ on $S$, the Hilbert series of $S/I$ coincides with that of $S/\ini_<(I)$.

Now, recall from \cite[Theorem 2.2.1]{HH} what the division algorithm is.  Let $f = u_0 - v_0$ and $g = u - v$ be binomials of $S$, where each of $u_0, v_0, u, v$ are monomials.  It may happen that either $v_0 = 0$ or $v = 0$.  Let $\ini_<(g) = u$.  
Suppose that either $u_0$ or $v_0$ is divisible by $u$ and, say, $u_0 = wu$ for some monomial $w$.  Then dividing $f$ by $g$ is the process
\[
f = u_0 - v_0 = wu - v_0 = w(g + v) - v_0 = wg + (wv - v_0)
\]
and the binomial $f_1 = wv - v_0 = u_1 - v_1$, where $u_1 = wv$ and $v_1 = v_0$, is called a remainder.  Let $f \rightarrow_g f_1$ denote the procedure of division.  If, furthermore, $f_1 \neq 0$ and if either $u_1$ or $v_1$ is divisible by $u$, then dividing $f_1$ by $g$ yields $f_2$.  Thus
\[
f \rightarrow_g f_1 \rightarrow_g f_2.
\]
A division algorithm guarantees that, after $N$ steps of divisions, a remainder $f_N = u_N - v_N$ possesses the property that either $f_N = 0$ or, neither $u_N$ nor $v_N$ is divisible by $u$.  Clearly the process of divisions
\[
f \rightarrow_g f_1 \rightarrow_g f_2 \rightarrow_g \cdots \rightarrow_g f_N
\]
is independent of the base field.  More generally, if $g_1, \ldots, g_s$ are binomials of $S$, one can divide $f$ by $g_1, \ldots, g_s$ and $N$ steps of divisions
\[
f \rightarrow_{g_{i_1}} f_1 \rightarrow_{g_{i_2}} f_2 \rightarrow_{g_{i_3}} \cdots \rightarrow_{g_{i_N}} f_N
\]
yield a binomial $f_N = u_N - v_N$ for which either $f_N = 0$ or, neither $u_N$ nor $v_N$ can be divisible by each of $\ini_<(g_1), \ldots, \ini_<(g_s)$.  We say that $f$ reduces to $f_N$ with respect to $g_1, \ldots, g_s$.  It may happen that another process of divisions
\[
f \rightarrow_{g_{i'_1}} f_1 \rightarrow_{g_{i'_2}} f_2 \rightarrow_{g_{i'_3}} \cdots \rightarrow_{g_{i'_{N'}}} f_{N'}
\]
yields a binomial $f_{N'}$ with $f_N \neq f_{N'}$ to which $f$ reduces with respect to $g_1, \ldots, g_s$.  Especially, one says that $f$ reduces to $0$ with respect to $g_1, \ldots, g_s$ if one can find a process of divisions of the form
\[
f \rightarrow_{g_{j_1}} f_1 \rightarrow_{g_{j_2}} f_2 \rightarrow_{g_{j_3}} \cdots \rightarrow_{g_{j_{N''}}} 0.
\]
Since the process of divisions is independent of the base field, it follows that the property that $f$ reduces $0$ with respect to $g_1, \ldots, g_s$ is independent of the base field.

We turn to the discussion of Buchberger criterion \cite[Theorem 2.3.2]{HH} together with Buchberger algorithm \cite[pp.~37--38]{HH}.  Let $I = (f_1, \ldots, f_s)$, where each $f_i = u_i - v_i$ is a binomial with $\ini_<(f_i)=u_i$.  The $S$-polynomial $S(f_i, f_j)$ is the binomial
\[
S(f_i, f_j) = u f_i - v f_j = v v_j - u v_i,
\]
where $u$ and $v$ are monomials with $uu_i = vu_j$ for which $uu_i = vu_j$ is the least common multiple of $u_i$ and $u_j$.  Buchberger criterion guarantees that $\{f_1, \ldots, f_s\}$ is a Gr\"obner basis of $I$ with respect to $<$ if and only if each of the $S$-polynomials $S(f_i, f_j)$ with $1 \leq i < j \leq s$ reduces to $0$ with respect to $f_1, \ldots, f_s$.  Thus, if $\{f_1, \ldots, f_s\}$ is {\em not} a Gr\"obner basis of $I$, then there is $S(f_i, f_j)$ which reduces to a binomial $f_{s+1} = u_{s+1} - v_{s+1} \neq 0$ with $\ini_<(f_{s+1})=u_{s+1}$ with respect to $f_1, \ldots, f_s$, where neither $u_{s+1}$ nor $v_{s+1}$ can be divisible by each of $u_1, \ldots, u_s$.  Since $f_{s+1}$ belongs to $I$, one may consider $I = (f_1, \ldots, f_s, f_{s+1})$.  If each of the $S$-polynomials $S(f_i, f_j)$ with $1 \leq i < j \leq s + 1$ reduces to $0$ with respect to $f_1, \ldots, f_s, f_{s+1}$, then $\{f_1, \ldots, f_s, f_{s+1}\}$ is a Gr\"obner basis of $I$ with respect to $<$.  Buchberger algorithm guarantees that, after a finite number of repeating the procedure of adding a remainder to a system of binomial generators of $I$, one has $I = (f_1, \ldots, f_s, f_{s+1}, \ldots, f_{s+q})$ for which each of the $S$-polynomials $S(f_i, f_j)$ with $1 \leq i < j \leq s + q$ reduces to $0$ with respect to $f_1, \ldots, f_s, f_{s+1}, \ldots, f_{s+q}$ and $\{f_1, \ldots, f_s, f_{s+1}, \ldots, f_{s+q}\}$ is a Gr\"obner basis of $I$ with respect to $<$.  The procedure of Buchberger algorithm, which consists of the computation of $S$-polynomials and their remainders, is independent of the base field.  Thus in particular, the property that a system of generators of $I$ is a Gr\"obner basis of $I$ with respect to $<$ is independent of the base field.

It then follows that $S/\ini_<(I)$ is independent of the base field.  Since the Hilbert function of a monomial ideal is independent of the base field, the Hilbert function of $S/I$ is independent of the base field, as desired.
\end{proof}

For a monomial ideal $I$ we denote by $G(I)$ the unique minimal set of monomial generators of $I$.

\begin{Lemma}
\label{monomialkoenig}
Let $I$ be a monomial ideal of height $h$. Then $I$ is of K\"onig type for any monomial order, if and only if there exists a regular sequence of monomials $u_1,\ldots,u_h$ in $G(I)$.
\end{Lemma}

\begin{proof} Suppose $I$ is of K\"onig type, and let $f_1,\ldots, f_h$ be homogeneous polynomials such that $u_1=\ini_<(f_1),\ldots,u_h=
\ini_<(f_h)$ is a regular sequence for some  monomial order $<$. Since $I$ is a monomial ideal, it follows that $u_i\in I$ for all $i$.
Hence there exist $v_i\in G(I)$ with $v_i\mid u_i$ for all $i$.  The sequence $v_1,\ldots,v_h$ is again a regular sequence.

The converse direction is trivial.
\end{proof}

\begin{Lemma}
\label{remainskoenig}
Let $I\subset S$ be a graded ideal and $<$ a monomial order on $S$. If $I$ is of K\"onig  type with respect to $<$, then $\ini_<(I)$ is of K\"onig type.
\end{Lemma}

\begin{proof}
Since $\dim S/I=\dim S/\ini_<(I)$ (see \cite[Theorem 3.3.4]{HH}), it follows that $h=\height I=\height \ini_<(I)$.  Suppose that $I$ is of K\"onig type. Then there exist homogeneous polynomials  $f_1,\ldots,f_h\in I$ such that $\ini_<(f_1),\ldots,\ini_<(f_h)$ is a regular sequence. Hence by Lemma~\ref{monomialkoenig},  $\ini_<(I)$ is  of K\"onig type.
\end{proof}

The converse of Lemma~\ref{remainskoenig} does not hold in general. Indeed if $\ini_<(I)$ is of K\"onig type, then always there exist $f_1,\ldots,f_h\in I$ of length $h=\height I$ such that $\ini_<(f_1),\ldots,\ini_<(f_h)$ is a regular sequence, but there may not exist such a sequence among the elements of a minimal generating set of $I$. See the following

\begin{Example}{\em
Consider the ideal $I=(x_1x_2-x_4^2,x_2x_3)$. Then with the reverse lex order $\ini_<(I)=(x_2x_3,x_1x_2,x_3x_4^2)$ which is an ideal of height $2$ and $x_1x_2,x_3x_4^2$ is the only regular sequence of length $2$ among its minimal generators. While $I$ is not of K\"onig type. Indeed any minimal generator of $I$ is of degree $2$, so its initial term is of degree $2$ as well. Thus there exists no sequence of length $2$ belonging to a minimal generating set of $I$ such that their initials form a regular sequence.}
\end{Example}

In \cite[Theorem 3.2]{HM} the first and third authors of this paper proved the following result.

\begin{Theorem}
\label{juergensomayeh}
Let $I\subset S$ be a monomial ideal. The following conditions are equivalent:
\begin{enumerate}
\item[{\em(a)}]  $I$ is of K\"onig type.
\item[{\em (b)}] $S/I$ admits  a special system of parameters.
\end{enumerate}
Moreover, if the equivalent conditions hold, then the sequence $C$ as given in (\ref{goodsequence}) is a special system of parameters.
\end{Theorem}

Combining Theorem~\ref{independent}, Lemma~\ref{remainskoenig} and  Theorem~\ref{juergensomayeh} we obtain

\begin{Corollary}
\label{initial}
Suppose $I$ is of K\"onig type with respect to $<$. Then the Cohen--Macaulayness of $\ini_<(I)$ does not depend on the base field.
\end{Corollary}

The following lemma is useful, when we deal with ideals of K\"onig type. Recall that the unmixed part of an ideal in a Noetherian ring is the intersection of its primary components of minimal height.

\begin{Lemma}\label{colon}
Let $J\subset I\subset S$ be ideals with $\height J=\height I=h$ such that $J$ is a complete intersection and a radical ideal. Then
\begin{enumerate}
\item[{\em (a)}]  $J:I=\bigcap_{P\in\min(J)\setminus \min(I)} P$. In particular $J:I$ is a radical ideal and unmixed.
\item[{\em (b)}]  $J:(J:I)$ is the unmixed part of $\sqrt{I}$.
\end{enumerate}
\end{Lemma}

\begin{proof}
(a) It is easy to see that $\bigcap_{P\in\min(J)\setminus \min(I)} P \subseteq J:I$. Now by contradiction suppose that there exists
$f\in (J:I)\setminus \bigcap_{P\in\min(J)\setminus \min(I)} P$. Then $f\notin P$ for some $P\in \min(J)\setminus \min(I)$. On the other hand $f/1\in (J:I)_P=J_P:I_P=PS_P:S_P=PS_P$. Thus $f\in P$, which is a contradiction.

(b) By (a), $$J:(J:I)=\bigcap_{P\in\min(J)\setminus (\min(J)\setminus \min(I))} P.$$
Since all minimal prime ideals in $\min(J)$ have height $h$, it follows that
$$\min(J)\setminus (\min(J)\setminus \min(I))=\min(J)\cap \min(I)=\{P\in \min(I): \height P=h\}.$$ This proves (b).

\end{proof}

\begin{Corollary}
Let $I$ be a graded ideal of K\"onig type with respect to $f_1,\ldots,f_h$ and assume that $\ini_<(f_i)$ is a squarefree monomial for all $i$.
Then the unmixed part of $\sqrt{I}$ is given by  $J:(J:I)$, where $J=(f_1,\ldots,f_h)$.
\end{Corollary}

\begin{proof}
Since $\ini_<(f_1),\ldots,\ini_<(f_h)$ is a regular sequence, it follows that $f_1,\ldots,f_h$ forms a Gr\"obner basis of $J$. Since $\ini_<(f_i)$ is a squarefree monomial for all $i$, it follows that $\ini_<(J)$ is a squarefree monomial ideal. This implies that $J$ is a radical ideal, see for example the proof of ~\cite[Corollary 2.2]{HHHKR}. Now by Lemma~\ref{colon} the result holds.
\end{proof}



\section{Properties of monomial  ideals of K\"onig type}
\label{two}
In this section we give a combinatorial characterization for Cohen-Macaulay monomial ideals of K\"onig type.
When $I$ is the edge ideal of a graph $G$, a description of Cohen-Macaulayness and the canonical module of $S/I(G)$, when $G$ is Cohen-Macaulay are given in terms of data from $G$.
We consider the following setting.

Let $I$ be a monomial ideal of  K\"onig type with respect to $u_1,\ldots,u_h\in G(I)$. Then $S/I$ has a special system of parameters $f_1,\ldots,f_d$ in the form of (\ref{goodsequence}) attached  to $u_1,\ldots,u_h$ (see the proof of \cite[Theorem 3.2]{HM}). For any $f_k=x_i-x_j$ we may assume that $i<j$.
For any $B\subseteq \{f_1,\ldots,f_d\}$, we set $I_B$ be the ideal obtained from $I$ by replacing $x_j$ by $x_i$ for any $x_i-x_j\in B$.
We call a monomial $g$ a modification of $f$ by $B$, if we  get $g$ when we replace $x_j$ by $x_i$ in $f$ for any $x_i-x_j\in B$.

The following theorem gives a combinatorial characterization for Cohen-Macaulay monomial ideals of K\"onig type and is a generalization of \cite[Proposition 3.2]{CRT}.

\begin{Theorem}\label{IB}
Let $I$ be a monomial ideal of  K\"onig type with a special system of parameters $f_1,\ldots,f_d$ attached  to $u_1,\ldots,u_h\in G(I)$. The following conditions are equivalent:
\begin{enumerate}
\item[{\em (i)}]  $S/I$ is Cohen-Macaulay.
\item[{\em (ii)}]  $S/I_B$ is Cohen-Macaulay for any $B\subseteq \{f_1,\ldots,f_d\}$.
\item[{\em (iii)}] $I_B$ is unmixed for any $B\subseteq \{f_1,\ldots,f_d\}$.
\end{enumerate}
\end{Theorem}

\begin{proof}
Without loss of generality we may assume that any variable $x_i$ belongs to the support of some $u_j$. Thus
each $f_k$ is of the form $x_i-x_j$ with $i<j$.

(i)\implies (ii): Let $B=\{g_1,\ldots,g_r\}\subseteq \{f_1,\ldots,f_d\}$. Since $R=S/I$ is a Cohen-Macaulay ring, $g_1,\ldots,g_r$ is a regular sequence, which implies that $R/(g_1,\ldots,g_r)R$ is Cohen-Macaulay. Notice that $R/(g_1,\ldots,g_r)R\cong T/I_B$, where $T=S/(x_j:\ j\in A)$ and $A=\{1\leq j\leq n:\ x_i-x_j\in B \textrm{ for some } i\}$. Thus $T/I_B$ is Cohen-Macaulay, and hence $S/I_B$ is Cohen-Macaulay.

(ii)\implies (iii): Any Cohen-Macaulay ideal is unmixed.

(iii)\implies (i): By contradiction assume that $R=S/I$ is not Cohen-Macaulay. Then $f_1,\ldots,f_d$ is not a regular sequence. So there exists
$k$ with $f_k\in Z(R/(f_1,\ldots,f_{k-1})R)$. Since  $R/(f_1,\ldots,f_{k-1})R\cong T/I_B$ for $B=\{f_1,\ldots,f_{k-1}\}$ and some polynomial ring $T$, we have
$f_k\in P$ for some minimal prime ideal $P$ of $I_B$. By assumption $u_1,\ldots,u_h$ is a regular sequence in $G(I)$ such that for each $i$, $\Supp(f_i)\subseteq \Supp(u_j)$  for some $1\leq j\leq h$, where $h=\height I$. If $f_k=x_i-x_j$, then $x_i,x_j\in P$, since $P$ is a monomial ideal. Let $u'_t$ be the modification of $u_t$ by $B$ for any $1\leq t\leq h$.
Then $(u'_1,\ldots,u'_h)+(x_i,x_j)\subseteq P$. Let $1\leq \ell\leq h$ be the unique integer with $x_ix_j|u_{\ell}$. Then
$u'_1,\ldots,u'_{\ell-1},u'_{\ell+1},\ldots,u'_h,x_i,x_j$ is a regular sequence in $P$. So
we obtain $\height P>h$. Notice that $\height I_B=\height I=h$. Since $I_B$ is unmixed, $h=\height I_B=\height P>h$, which is a contradiction.
\end{proof}

For a graph $G$ we show the set of all minimal vertex covers of $G$ by $\min(G)$ and $|\min(G)|$ is denoted by $c_G$.
For a minimal vertex cover $C\subseteq V(G)=\{x_1,\ldots,x_n\}$ of $G$, we set $L_C=(x_i:\ x_i\in C)$.
A perfect matching of K\"onig type for a graph $G$, is a  matching $\{e_1,\ldots,e_h\}$ with $h=\tau(G)$ and $V(G)=\cup_{i=1}^h e_i$.
It is known that any unmixed K\"onig graph without isolated vertices has a perfect matching of K\"onig type, see for example \cite[Proposition 15]{CCR}.
\begin{Theorem}
\label{canonical2}
Let $G$ be an Cohen-Macaulay K\"onig graph and $\{e_1,\ldots,e_n\}$ be a perfect matching of K\"onig type for $G$. Let $R=S/I(G)$ and $H$ be a graph with $V(H)=V(G)$ and
$$E(H)=\{\{z,w\}:\ z\in e_i,\ w\in e_j,\ i\neq j,\ (e_i\setminus\{z\})\cup (e_j\setminus\{w\})\in E(G)\}.$$ Then
\begin{enumerate}
\item[{\em (a)}]  $\omega_R\iso \overline{I(H)^{\vee}}$, where $\overline{I}$ denotes the image of the ideal $I$ in $R$ under the canonical epimorphism $S\to R$.
\item[{\em (b)}] $\type(R)=c_H$.
\end{enumerate}
\end{Theorem}

\begin{proof}
(a) We use a basic fact from linkage theory (first observed by \cite{PS}): let $I\subset  S$ be a graded Cohen--Macaulay ideal,  $J\subset I$ a complete intersection with $\height I=\height J$, and let $R=S/I$. Then
\[
\omega_R\iso (J:I)/J.
\]
For any $1\leq i\leq n$, let $e_i=\{x_i,y_i\}$. Set $J=(x_1y_1,\ldots,x_ny_n)$ and let $T$ be the graph with the vertex set $V(G)$ and the edge set $E(T)=\{e_1,\ldots,e_n\}$. Then $J=I(T)$. Note that any minimal vertex cover $C$ of $G$ has cardinality $\tau(G)=n$ and for any $1\leq i\leq n$, $|e_i\cap C|=1$.
So $\min(G)\subseteq \min(T)$. By Lemma~\ref{colon},
\begin{eqnarray}
\label{minsect}
J:I(G)=\Sect_{C\in\min(T)}L_C:\Sect_{C\in\min(G)}L_C=  \Sect_{C\in\min(T)\setminus \min(G)}L_C.
\end{eqnarray}


We have $C\in \min(T)\setminus \min(G)$, if and only if $C\in \min(T)$ and there exists an edge $e\in E(G)\setminus E(T)$ such that $C\cap e=\emptyset$. Let $C\in \min(T)\setminus \min(G)$ and $e=\{z,w\}\in E(G)\setminus E(T)$
with $C\cap e=\emptyset$. Since $V(G)=\cup_{i=1}^n e_i$, we have $z\in e_i$ and $w\in e_j$ for some $i$ and $j$. Clearly $i\neq j$ and $(e_i\setminus\{z\})\cup (e_j\setminus\{w\})\subseteq C$. In the sequel we denote the set $(e_i\setminus\{z\})\cup (e_j\setminus\{w\})$ by $\{\tilde{z},\tilde{w}\}$. Then we conclude that $\min(T)\setminus \min(G)=\bigcup A_{z,w}$, where the union is taken over
$\{z,w\}\in E(G)\setminus E(T)$ and $$A_{z,w}=\{D\cup \{\tilde{z},\tilde{w}\}: D\in \min(T\setminus \{e_i,e_j\}), z\in e_i, w\in e_j\}.$$
So $$\Sect_{C\in\min(T)\setminus \min(G)}L_C=\Sect_{\{z,w\}\in E(G)\setminus E(T)}(\Sect_{C\in A_{z,w}}L_C).$$

Note that for $z\in e_i$ and $w\in e_j$, we have $$\Sect_{C\in A_{z,w}}L_C=(\tilde{z},\tilde{w})+I(T\setminus \{e_i,e_j\}).$$
Therefore
\begin{eqnarray}
\label{innersec}
\Sect_{C\in\min(T)\setminus \min(G)}L_C=(\Sect_{\{z,w\}\in E(G)\setminus E(T)}(\tilde{z},\tilde{w}))+I'=I(H)^{\vee}+I',
\end{eqnarray}
for some ideal $I'\subseteq J\subseteq I(G)$.
By second isomorphism Theorem,
\begin{eqnarray}
\label{second iso}
(\Sect_{C\in\min(T)\setminus \min(G)}L_C)/J\cong (\Sect_{C\in\min(T)\setminus \min(G)}L_C+I(G))/I(G).
\end{eqnarray}
Now by (\ref{minsect}), (\ref{innersec}) and (\ref{second iso}) we get
$$\omega_R\iso (\Sect_{C\in\min(T)\setminus \min(G)}L_C+I(G))/I(G)\cong \overline{I(H)^{\vee}}.$$

(b) By (a) we get $\omega_R\iso (\Sect_{C\in\min(T)\setminus \min(G)}L_C)/J=(I(H)^{\vee}+I')/J$, where $I'\subseteq J$. So $\omega_R\iso (I(H)^{\vee}+J)/J$ and then $\type(R)$ is equal to the number of minimal generators of $(I(H)^{\vee}+J)/J$. Thus it is enough to show that this number is equal to the number of minimal generators $I(H)^{\vee}$.
Let $\{u_1,\ldots,u_m\}$ be the minimal generating set of monomials of $I(H)^{\vee}$, where $u_i=\xb_{C_i}$ for each $i$ and $\{C_1,\ldots,C_m\}=\min(H)$. Here by
$\xb_C$ we mean $\prod_{x_i\in C} x_i$. Let $\overline{u}_i=u_i+J$ for all $i$. We show that
$\{\overline{u}_1,\ldots,\overline{u}_m\}$ is a minimal generating set of $(I(H)^{\vee}+J)/J$. Clearly it is a generating set. By contradiction assume that it is not minimal. Then
$u_j-\sum_{i\neq j} r_iu_i\in J$ for some $1\leq j\leq m$ and $r_i\in S$. Since $J$ is a monomial ideal and none of the $u_i$'s with $i\neq j$ divides $u_j$, we have
$u_j\in J$. This means that $x_ky_k|u_j$ for some $1\leq k\leq n$, or equivalently $x_k,y_k\in C_j$. Thus $N_H(x_k)\nsubseteq C_j$, otherwise $C_j\setminus \{x_k\}$ is a vertex cover of $H$ as well, which contradicts to the minimality of $C_j$. Similarly $N_H(y_k)\nsubseteq C_j$. So there exist $z\in N_H(x_k)\setminus C_j$ and $w\in N_H(y_k)\setminus C_j$. Then $\{x_k,z\},\{y_k,w\}\in E(H)$. Let $z\in e_r=\{z,z'\}$ and $w\in e_s=\{w,w'\}$. Then by the definition of $H$ we have $\{y_k,z'\},\{x_k,w'\}\in E(G)$. Since $G$ is an unmixed K\"onig graph, by ~\cite[Corollary 2.11]{MRV}, we should have $\{z',w'\}\in E(G)$. This means that $\{z,w\}\in E(H)$. But since $z,w\notin C_j$, this contradicts to $C_j$ be a vertex cover of $H$.
\end{proof}

We use Theorem~\ref{juergensomayeh} to derive a combinational Cohen--Macaulay criterion for K\"onig graphs. The reader may compare our criterion
 with \cite[Proposition 28]{CCR} and \cite[Theorem 2.7]{HM}.

Let $G$ be K\"onig graph on $[n]$, and let $\{e_1,\ldots,e_m\}$ be a matching of $G$  with $m=\tau(G)$.  Let $e_k=\{i_k,j_k\}$ with $i_k<j_k$ for $k=1,\ldots,m$. We define a new graph $G_0$ with $V(G_0)=\{i_1,\ldots,i_m\}$  for which $e=\{i_k,i_l\}\in E(G_0)$ if and only if $e_k$ and $e_l$ are adjacent in $G$, i.e., if and only if there exists an edge $e'\in G$ such that $e_k\sect e'\neq\emptyset$ and $e_l\sect e'\neq \emptyset$.

Part (b) of the following result is due to \cite[Corollary 4.4]{CRT}. For the convenience of the reader we include its short proof.

\begin{Theorem}
\label{characterizekoenig}
Let $G$ be a K\"onig graph. Let $\alpha(G)$ be the number of vertex covers  $D$ of $G$ with $|D|=\tau(G)$, and let $\Delta_{G_0}$ be the independence  complex of $G_0$.  With the assumptions and notation introduced we have:
\begin{enumerate}
\item[{\em (a)}] $S/I(G)$ is Cohen--Macaulay if and only if $\alpha(G)=|\Delta_{G_0}|$.
\item[{\em (b)}] If $S/I(G)$ is Cohen--Macaulay, then $\type(S/I(G))$ is the number of facets of $\Delta_{G_0}$ and $\reg I(G)= \dim \Delta_{G_0}+1$.
\end{enumerate}
\end{Theorem}

\begin{proof}
(a) We notice that $\alpha(G)$ coincides with the number of minimal prime ideals $P$ of $I(G)$ with $\height P=\height I(G)$. This shows that $\alpha(G)=e(R)$,  where $R=S/I(G)$. By Theorem~\ref{juergensomayeh}, the sequence $C$ in (\ref{goodsequence}) is a system of parameters for $R$. Let $L$ be the ideal generated by this sequence. Then $R/LR\iso K[\Delta_{G_0}]/(x_1^2,\ldots,x_n^2)$. Thus the set of monomials $\{\prod_{i\in F}x_i\:\; F\in \Delta_{G_0}\}$  is a $K$-basis of $R/LR$. It follows that $\length(R/LR)=|\Delta_{G_0}|$. Here $\length(M)$ denotes the length of a module $M$. As observed in Theorem~\ref{independent}, $R$ is Cohen--Macaulay if and only if $e(R)=e(R/LR)$. Since $\dim R/LR=0$ we have $e(R)=\length(R/LR)$. Thus, the desired conclusion follows.

(b) If $R$ is Cohen--Macaulay, then $\type(R)=\type(R/LR)$. Since $\dim R/LR=0$ it follows that $\type(R/LR)=\dim_K \Soc(R/LR)$. Here $\Soc(R/LR)$ denotes the socle of $R/IR$. It follows from the structure of $R/LR$, as described in part (a) of the proof, that the monomials $\prod_{i\in F}x_i$ with $F$ a facet of $\Delta_{G_0}$ form a  $K$-basis of $\Soc(R/LR)$. This shows that $\type(I(G))$ is the number of facets of $\Delta_{G_0}$. Keeping in mind that $\reg(R)$ coincides with the maximal degree of a socle element of $R/LR$, we see that $\reg I(G)= \dim \Delta_{G_0}+1$.
\end{proof}

We end this section by proving a property for an unmixed monomial ideal of K\"onig type generated in one degree.

\begin{Remark}\label{rem1}
{\em
Let $I$ be a monomial ideal and $I^\wp$ be its polarization. For a monomial $u=\prod_{i=1}^n x_i^{a_i}$, let $u^\wp=\prod_{i=1}^n \prod_{j=1}^{a_i} x_{ij}$.  Then it is easy to see that $u_1,\ldots,u_h$ is a regular sequence in $I$ if and only if $u_1^\wp,\ldots,u_h^\wp$ is a regular sequence in $I^\wp$. So from the equality $\height I=\height I^\wp$,
we deduce that $I$ is of K\"onig type if and only if $I^\wp$ is of K\"onig type.}
\end{Remark}

A squarefree monomial ideal $I$ generated in degree $d$ is called \textit{very well-covered}, if $|\bigcup_{u\in G(I)} \Supp(u)|=(\height I)d$.

\begin{Theorem}\label{very}
Let $I$ be an unmixed monomial ideal generated in degree $d$. If $I$ is of K\"onig type, then $I^\wp$ is very well-covered.
\end{Theorem}

\begin{proof}
Note that $I$ is unmixed if and only if $I^\wp$ is unmixed (see \cite[Corollary 2.6]{Fa}). So in view of Remark~\ref{rem1} we may assume that $I$ is squarefree.
Let $I$ be of K\"onig type and set $h=\height I$. Then there exists a regular sequence $u_1,\ldots,u_h\in G(I)$.
Set $A=\bigcup_{u\in G(I)} \Supp(u)$ and by contradiction assume that $hd\neq |A|$. Since $u_1,\ldots,u_h$ is a regular sequence of squarefree monomials, this implies that $hd<|A|$. Then there exists  $x_{\ell}\in A$ which does not divide $u_1,\ldots,u_h$. Then $x_{\ell}$ divides $v$ for some $v\in G(I)$. So there exists a minimal prime ideal $P$ of $I$ which contains $x_{\ell}$. Therefore $(x_{\ell},u_1,\ldots,u_h)\subseteq P$, which implies that $\height P\geq h+1$. This is a contradiction, since $I$ is unmixed. So $|A|=hd$ as desired.
\qed
\end{proof}

The converse of Theorem~\ref{very} is true in the case that $I$ is generated in degree $2$.  Indeed, since $I$ is unmixed, then $I^\wp$ is unmixed as well. So by \cite[Lemma 2.1]{CRT}, if $I^\wp$ is very well-covered, then $I^\wp$ is of K\"onig type. Therefore,  Remark~\ref{rem1} implies that $I$ is of K\"onig type.
But this is not the case in general.
To see this, consider the ideal $I=(x_1x_2x_3,x_1x_3x_4,x_1x_4x_6,x_3x_4x_5)$ in $S=K[x_1,\ldots,x_6]$. Then $I$ is unmixed with $\height I=2$, $d=3$ and $n=6=(\height I)d$,. So $I$ is unmixed and very well-covered, while any regular sequence of monomials in $I$ has length $1$. So $I$ is not of K\"onig type.

\section{Binomial edge ideals of K\"onig type}
\label{three}

Let  $K$  be a field  and  $S=K[x_1,\ldots,x_n,y_1,\ldots,y_n]$ the polynomial ring over $K$ in the variables $x_1,\ldots, y_n$. For $1\leq i<j\leq n$ we set  $f_{ij}=x_iy_j-x_jy_i$.

Let $G$ be a simple graph on the vertex set $[n]$ with set of edges $E(G)$. The ideal $J_G=(f_{ij}\: i<j, \{i,j\}\in E(G))$ is the  {\em  binomial edge ideal} of $G$.
Throughout this section $<$ denotes the lex order induced by $x_1>\cdots>x_n>y_1>\cdots>y_n$.

\begin{Lemma}
\label{minimaledge}
Let $G$ be a graph on $[n]$. Then the set $\{f_{ij}\: i<j, \{i,j\}\in E(G)\}$ is a minimal set of generators of $J_G$.
\end{Lemma}

\begin{proof}
The initial monomials of the generators $f_{ij}$ are pairwise distinct. This yields the desired conclusion.
\end{proof}

The binomials which belong to a minimal set of generators  of $J_G$ are described in the following

\begin{Lemma}
\label{minimalgenerators}
Let $G$ be a graph on $[n]$  and $f\in J_G$ a binomial.  If $f$ is part of a minimal set of generators of $J_G$, then the exists and edge $\{i, j\}$ of $G$  with $i<j$ such that $f=\pm f_{ij}$.
\end{Lemma}

\begin{proof}
Let $\epsilon_i$ be the $i$th canonical unit vector of $\ZZ^n$. We set $\deg(x_i)=\deg(y_i)=\epsilon_i$. Each $f_{ij}$ is homogeneous of degree $\epsilon_i+\epsilon_j$, and hence $J_G$ is a $\ZZ^n$-graded ideal.

Let $f=u-v\in J_G$, then the total degree of $f$ is $2$.  Otherwise,  $f\in\mm J_G$, where $\mm$ is the graded maximal ideal of $S$, which by Nakayama's lemma is a contradiction. It follows that $f$ is a $K$-linear combination of the binomials $f_{ij}$. Hence,  the monomials $u$ and $v$ are of the form $x_ky_l$.

Since $J_G$ is $\ZZ^n$-graded, all $\ZZ^n$-graded components belong to $J_G$. Suppose  $u$ and $v$ have different $\ZZ^n$-degree, then $u,v\in J_G$. If we apply the map $\varphi$ which replaces each  $y_i$ by $x_i$, then $\varphi(J_G)=0$ while $\varphi(u)\neq 0$, a contradiction. It follow that $u$ and $v$ have the same $\ZZ^n$ degree, say $\deg(u)=\deg(v)=\epsilon_i+\epsilon_j$. This implies that $u,v\in \{x_iy_j,x_jy_i\}$. We may assume that $u=x_iy_j$. Then $v=x_jy_i$, since $f\neq 0$. Therefore, $f=\pm f_{ij}$.
Now we show that if  $f_{ij}\in J_G$, then $\{i,j\}$ is an edge of $G$. The set $\{x_ky_l\:\{k,l\} \in E(G)\}$ is the set of monomials of degree $2$ of $\ini_<(J_G)$.  Since $f_{ij}\in J_G$, it follows that $x_iy_j$ belongs to the ideal generated by this set of monomials. This is only possible, if $x_iy_j$ belongs to this set, and this implies that $\{i,j\}\in E(G)$.
\end{proof}

A \emph{path graph} on the vertex set $\{i_1,\ldots,i_n\}=[n]$ is a graph with the edge set $\{\{i_k,i_{k+1}\}:\ 1\leq i\leq n-1\}$ and is denoted by $P_n$. We call the vertices $i_1$ and $i_n$ the {\em endpoints} and $n-1$ the {\em length} of the path.
Since the edge set of a path graph can be seen from its vertex set, we usually denote a path graph by  $P_n: i_1,\ldots,i_n$.

For a graph $G$ and $i,j\in V(G)$, a {\em path from $i$ to $j$} (a path connecting $i$ and $j$) in $G$ is a subgraph of $G$, which is a path graph with the endpoints $i$ and $j$. We call a subgraph $P$ of  $G$ a {\em semi-path} if each component of $P$ is a path  graph. By definition, the length of the semi-path is its number of edges.

\begin{Lemma}
\label{semi-path}
$G$ is a semi-path  if and only if $J_G$ is a complete intersection.
\end{Lemma}

\begin{proof}

Suppose that $G$ is  a semi-path graph. We want to show that $J_G$ is a complete intersection. This is the case if the binomial edge ideal of each component of $G$ is a complete intersection.  We may therefore assume that $G$ is connected and that  $\{i,i+1\}$ for $i=1,\ldots, n$  are the edges of $G$. Then
$J_G=(f_{1,2},\ldots, f_{n,n+1})$. We consider the lexicographic order on $S$ induced by $x_1>x_2>\cdots >x_{n+1}>y_1>\cdots>y_{n+1}$. Then $\ini_<(f_{i,i+1})=x_iy_{i+1}$ for all $i$. Since $x_1y_2,\ldots, x_{n}y_{n+1}$ is a regular sequence, it  follows  $f_{1,2},\ldots, f_{n,n+1}$ is a regular sequence.

Conversely, assume that $J_G$ is a complete intersection. Then $J_G=(g_1,\ldots, g_{n})$ where $g_1,\ldots,g_{n}$
is a regular sequence. Then $\height  J_G=n$ and $J_G$ is minimally generated by $n$ elements. Since the set of binomials
$\{f_{ij}\:\; \{i,j\}\in E(G),\ i<j\}$ is a minimal set of generators of $J_G$ it follows that the elements of this set form a regular sequence. Thus we may assume that  $g_k=f_{i_kj_k}$ with $\{i_k, j_k\}\in E(G)$ and $i_k<j_k$  for $k=1,\ldots,n$.

We proceed by induction on the number of edges of $G$. If $G$ has only one or two edges, then  $G$ is a semi-path and there is nothing to prove. So now assume that $G$ have more than two edges. Then $n\geq 3$. Since $g_1,\ldots,g_{n-1}$ is also a complete intersection, our induction hypothesis implies that the subgraph $G'$ of $G$ with  $E(G')=\{\{i_k, j_k\}\: k=1,\ldots,n-1\}$ is a semi-path. If the edge
$\{i_{n},j_{n}\}$ is a component of $G$ or if $i_{n}$ or $j_{n}$ is an endpoint of  one of the components of $G'$, then $G$ is a semi-path. Suppose these cases  do not occur. Then there exists and integer $k$ such  that $j_{k}=i_{k+1}$ and $i_{n}=j_k$ or $j_{n}=j_k$. Assume that $i_n=j_k$.  Then
$g_n=x_{i_n}y_{j_n}-x_{j_n}y_{i_n}$,  $g_k=x_{i_k}y_{i_n} -
x_{i_n}y_{i_k}$ and $g_{k+1}= x_{i_n}y_{j_{k+1}}-
x_{j_{k+1}}y_{i_n}$. This implies that $g_k,g_{k+1},g_n\in (x_{i_n},y_{i_n})$ which is a contradiction since $g_k,g_{k+1},g_n$ is a regular sequence which implies that $\height( g_k,g_{k+1},g_n )=0$,  while $\height (x_{i_n},y_{i_n})=2$. The case $j_n=j_k$ is treated similarly.
\end{proof}

\begin{Corollary}\label{corm}
Let $G$ be a simple graph. Then the binomial edge ideal $J_G$ is  of {\em K\"onig type}, if and only if $J_G$ contains a regular sequence $g_1,\ldots,g_h$ such that each $g_k$ is a binomial of the form $f_{ij}$, where $\{i,j\}\in E(G)$ and $h=\height J_G$.
\end{Corollary}

\begin{proof}
By Lemma~\ref{minimalgenerators}, the minimal binomial generators of $J_G$ are of the form $f_{ij}$,  where $\{i,j\}\in E(G)$ and by Lemma~\ref{semi-path} such a sequence is a regular sequence if and only if the  initial monomials of this sequence form a regular sequence. This yields the desired conclusion.
\end{proof}

\begin{Theorem}
\label{lenghtsemipaths}
Let $G$ be a graph on $[n]$ and $d=\dim S/J_G$.
\begin{enumerate}
\item[{\em (a)}] Let $P$  be a semi-path in $G$ of length $r$.  Then $r\leq 2n-d$.

\item[{\em (b)}] The following conditions are equivalent:

\begin{enumerate}
\item[{\em (i)}]  $J_G$ is of  K\"onig type.
\item[{\em (ii)}] There exists a  semi-path in $G$ of length $2n-d$.
\end{enumerate}
\end{enumerate}
\end{Theorem}

\begin{proof}
(a) By Lemma~\ref{semi-path},  $J_P$ is generated by a regular sequence whose length is $r$. It follows that $r=\height J_P\leq \height J_G= 2n-d$.

(b) First note that by Lemma~\ref{minimalgenerators}, if $f_{ij}\in J_G$, then $\{i,j\}$ is an edge of $G$.
Let $h=\height J_G=2n-d$. By Corollary~\ref{corm}, $J_G$ is of K\"onig type if and only if $J_G$ contains a regular sequence $g_1,\ldots,g_h$ where each $g_k$ is a of the form $f_{ij}$ for some $\{i,j\}\in E(G)$.
Now, by Lemma~\ref{semi-path},  $g_1,\ldots g_h$ is a regular sequence in $J_G$, if and only if the subgraph $P$ of $G$ whose edges correspond to the leading terms of
the $g_i$'s is a semi-path of length $h$.
This proves the equivalence of (i)\iff (ii).
\end{proof}

Recall that a subgraph $H$ of $G$ is called a {\em spanning subgraph}, whenever $V(H)=V(G)$.   A connected graph $G$ is called  {\em traceable}, if it has a spanning path $P$ as a subgraph. More generally, $G$ is called traceable if each of its connected components is traceable.

\begin{Proposition}
\label{somayehchanged}
Let $G$ be a graph.
\begin{enumerate}
\item[{\em (a)}] If $G$ is traceable,  then $J_G$ is of K\"onig type. In particular, closed graphs are of K\"onig type.

\item[{\em (b)}] If $J_G$ unmixed and of K\"onig type,  then $G$ is traceable.
\end{enumerate}
\end{Proposition}

\begin{proof} Let $V(G)=[n]$. Without loss of generality we may assume that $G$ is connected.

(a) Since $G$ is traceable, there exists a path $P$ in $G$ with $V(P)=[n]$. Then by Lemma~\ref{semi-path},  the ideal $J_P\subset J_G$ is generated  by a regular sequence of length $n-1$ of binomials of the form $f_{ij}$. In particular, $\height J_G\geq n-1$.  Since $G$ is connected, we have $\dim S/J_G\geq n+1$ (see \cite[Corollary 3.3]{HHHKR}). Therefore, $\height J_G\leq n-1$. Hence we conclude that $\height J_G=n-1$. This shows that $J_G$ is of K\"onig type.
The second statement follows from the observation that any closed graph is traceable.

(b) By Theorem~\ref{lenghtsemipaths},  there exists a  semi-path $P$ in $G$ of length $r=2n-d$, where $d=\dim S/J_G$. Since $J_G$ is unmixed
we have $d=n+1$,  see proof of \cite[Corollary 3.4]{HHHKR}. It follows that $r=n-1$. Any semi-path of $G$ of length $n-1$ is indeed a path. This implies that $P$ is a path of length $n-1$ and hence is a spanning path of $G$.
\end{proof}

There are  graphs which are neither closed, nor bipartite, nor chordal, nor traceable,  but whose binomial  edge ideal is of K\"onig type.  The graph $G$ in Figure~\ref{better} is such a graph. Indeed, it can be seen that $\dim S/J_G=9$, so that $2n-d=5$ which is the length of the path  $P: 1,2,3,4,5,6$.  On the other hand, it cannot be unmixed because it is not traceable.

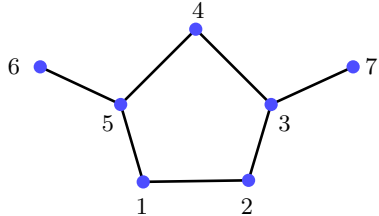
\begin{figure}[hbt]
\begin{center}
\newrgbcolor{ududff}{0.30196078431372547 0.30196078431372547 1.}
\psset{xunit=0.5cm,yunit=0.5cm,algebraic=true,dimen=middle,dotstyle=o,dotsize=5pt 0,linewidth=2.pt,arrowsize=3pt 2,arrowinset=0.25}
\begin{pspicture*}(3.,2.)(13.,9.)
\psline[linewidth=1.pt](8.,8.)(6.,6.)
\psline[linewidth=1.pt](6.,6.)(6.6,3.94)
\psline[linewidth=1.pt](8.,8.)(10.,6.)
\psline[linewidth=1.pt](10.,6.)(9.4,3.98)
\psline[linewidth=1.pt](6.6,3.94)(9.4,3.98)
\psline[linewidth=1.pt](10.,6.)(12.18,7.)
\psline[linewidth=1.pt](6.,6.)(3.86,7.)
\begin{scriptsize}
\psdots[dotstyle=*,linecolor=ududff](8.,8.)
\psdots[dotstyle=*,linecolor=ududff](6.,6.)
\psdots[dotstyle=*,linecolor=ududff](6.6,3.94)
\psdots[dotstyle=*,linecolor=ududff](10.,6.)
\psdots[dotstyle=*,linecolor=ududff](9.4,3.98)
\psdots[dotstyle=*,linecolor=ududff](12.18,7.)
\psdots[dotstyle=*,linecolor=ududff](3.86,7.)
\rput[ll] (5.5, 5.5) {$5$}
\rput[ll] (10.2, 5.5) {$3$}
\rput[ll] (7.9,8.5) {$4$}
\rput[ll] (6.4, 3.3) {$1$}
\rput[ll] (9.2, 3.3) {$2$}
\rput[ll] (3, 7) {$6$}
\rput[ll] (12.5, 7) {$7$}
\end{scriptsize}
\end{pspicture*}
\end{center}
\caption{$5$-cycle  with whiskers}
\label{better}
\end{figure}

Consider the graph $G$ in Figure~\ref{betternew}. The semi-path $P$ with the edges $\{1,2\},\{2,3\},\\ \{3,4\},\{5,6\}$ has length 4 and it is a semi-path in $G$ of largest length in $G$. On the other hand,  $\dim S/J_G=7$.
Since  $2n-d=12-7=5>4$, $J_G$ is not of K\"onig type.

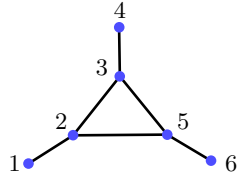
\begin{figure}[hbt]
\begin{center}
\newrgbcolor{ududff}{0.30196078431372547 0.30196078431372547 1.}
\psset{xunit=0.4cm,yunit=0.4cm,algebraic=true,dimen=middle,dotstyle=o,dotsize=4pt 0,linewidth=1.6pt,arrowsize=3pt 2,arrowinset=0.25}
\begin{pspicture*}(6.,3.)(14.,11.)
\psline[linewidth=1.pt](10.,8.)(8.447440930048627,6.0400146440895695)
\psline[linewidth=1.pt](8.447440930048627,6.0400146440895695)(11.573395339943344,6.063518060705319)
\psline[linewidth=1.pt](11.573395339943344,6.063518060705319)(10.,8.)
\psline[linewidth=1.pt](10.,8.)(9.975163010072361,9.636037386299284)
\psline[linewidth=1.pt](8.447440930048627,6.0400146440895695)(6.966725683256391,5.0998779794595785)
\psline[linewidth=1.pt](11.573395339943344,6.063518060705319)(13.030607170119827,5.193891645922578)
\begin{scriptsize}
\psdots[dotstyle=*,linecolor=ududff](10.,8.)
\psdots[dotstyle=*,linecolor=ududff](8.447440930048627,6.0400146440895695)
\psdots[dotstyle=*,linecolor=ududff](11.573395339943344,6.063518060705319)
\psdots[dotstyle=*,linecolor=ududff](9.975163010072361,9.636037386299284)
\psdots[dotstyle=*,linecolor=ududff](6.966725683256391,5.0998779794595785)
\psdots[dotstyle=*,linecolor=ududff](13.030607170119827,5.193891645922578)
\rput[ll](6.3, 5.1){$1$}
\rput[ll](7.85, 6.5){$2$}
\rput[ll](9.8,10.2){$4$}
\rput[ll](11.9, 6.5){$5$}
\rput[ll](9.2, 8.3){$3$}
\rput[ll](13.5, 5.1){$6$}
\end{scriptsize}
\end{pspicture*}
\end{center}
\caption{Triangle with whiskers}
\label{betternew}
\end{figure}

The following theorem shows the existence of special system of parameters for binomial edge ideals of K\"onig type.

\begin{Theorem}
\label{parametersystem}
Let $G$ be a graph such that $J_G$ is of
K\"onig type. Then $S/J_G$ has a special system of parameters of the form (\ref{goodsequence}).

\end{Theorem}

\begin{proof}
Let $V(G)=[n]$. By Theorem~\ref{lenghtsemipaths}, $G$ has a semi-path  $P$ of length $h=2n-d$,  where $d=\dim S/J_G$.
Let
\[
E(P)=\{\{i_1,j_1\},\ldots,\{i_h, j_h\}:\ i_{\ell}<j_{\ell}\}.
\]
Then $J_G$ is of K\"onig type with respect to $f_{i_1j_1},\ldots,f_{i_hj_h}$.
We show that the sequence $C$ attached to $f_{i_1j_1},\ldots,f_{i_hj_h}$ as in (\ref{goodsequence}) is a system of parameters of $S/J_G$.
Since $|C|=d$, it suffices to show that
$\dim S/(J_G+L)=0$, where $L=(C)$.

Since $J_P\subseteq J_G$ it suffices to show that $\dim S/(J_P+L)=0$.
Let $P_1, \ldots, P_k$ be the components of $P$. Without loss of generalities we may assume $P_1: 1,\ldots,a_1, P_2: a_1+1,\ldots,a_2, \ldots , P_i: a_{i-1}+1,\ldots,a_i, \ldots, P_k: a_{k-1}+1,\ldots, n$, where $a_{i-1}<a_{i}$ for $i=2,\ldots,k$ and $a_k=n$.
Then $J_P+L=\sum_{s=1}^k (J_{P_s}+L_s)$, where
\[
L_s=(x_i-y_j:\ \{i,j\}\in E(P_s), i<j)+(x_{a_s},y_{a_{s-1}+1}).
\]
Note that $J_{P_1}+L_1,\ldots,J_{P_k}+L_k$ are ideals in the polynomial rings with pairwise disjoint sets of variables. This implies that $S/(J_P+L)\cong \bigotimes_{i=1}^k S_i/(J_{P_i}+L_i)$, where $S_i=K[x_j,y_j: j\in V(P_i)]$.
Hence it is enough to show that $\dim S/(J_P+L)=0$, when $P$ is a path and $S=K[V(P)]$.
Let $P:1,\ldots, n$ be a path graph. Then $x_iy_{i+1}-x_{i+1}y_i\in J_P$ for $i=1,\ldots,n-1$ and $L=(x_i-y_{i+1}: 1\leq i\leq n-1)+(x_n,y_1)$.
Then $S/(J_P+L)\cong K[y_1,\ldots,y_n]/H$, where $H=(y_{i+1}^2-y_iy_{i+2}:\ 2\leq i\leq n-2)+(y_2^2,y_n^2)$.
By induction assume that $y_i^r\in H$ for some positive integer $r$. Then $y_{i+1}^{2r}-y_i^ry_{i+2}^r\in H$, because it is divided by $y_{i+1}^2-y_iy_{i+2}$.
Thus $y_{i+1}^{2r}\in H$. Since $y_2^2\in H$, this argument shows that $y_i^{2^{i-1}}\in H$ for all $i$.  So $\dim K[y_1,\ldots,y_n]/H=0$, which completes the proof.
\qed
\end{proof}

As a consequence of Theorems~\ref{independent} and ~\ref{parametersystem} we recover with a different proof  the following result, due to ~\cite{BMS}. In fact by Proposition~\ref{somayehchanged}
an unmixed binomial edge ideal $J_G$ is of K\"{o}nig type if and only if $J_G$ is traceable.

\begin{Corollary}
\label{independentofK}
Let $G$ be graph such that $J_G$ is of K\"{o}nig type. Then the Cohen--Macaulay property of $S/J_G$ does not depend on the base field $K$.
\end{Corollary}


We recall a few facts from \cite{HHHKR}. Let $G$ be a graph on $[n]$, and let $T\subset [n]$. We denote by $c(T)$ the number of connected components of $G_{[n]\setminus T}$. The set $T$ is called a {\em cut set} of $G$ if $c(T\setminus \{i\})<c(T)$ for all $i\in T$. We consider $T=\emptyset$ also as a cut set of $G$, and denote by $\cut(G)$ the set of all cut sets of $G$.

Let $T\subset [n]$, and let $G_1,\ldots, G_{c(T)}$ be the connected components of $G_{[n]\setminus T}$. We set
\[
P_T=(\Union_{i\in T}\{x_i,y_i\}, J_{\tilde{G}_1},\ldots, J_{\tilde{G}_{c(T)}}),
\]
where $\tilde{G}_i$ denotes the complete graph on the vertex set of $G_i$.  Then  $P_T$ is a minimal prime ideal of $J_G$  if and only if $T$ is a cut set,  and one has
\begin{eqnarray}
\label{cutsect}
J_G=\Sect_{T\in \cut(G)}P_T.
\end{eqnarray}

Suppose that $J_G$ is unmixed and $G$ is traceable. Let $P\subset G$ be path with $V(P)=V(G)$. Since $J_P$ is a complete intersection, it is Cohen--Macaulay and hence unmixed. If $h=\height J_G$, then all minimal prime ideals of $J_P$ and $J_G$ are of height $h$. Hence if $P_T$ is a minimal prime ideal of $J_G$, then  it is also a minimal prime ideal of $J_P$. This shows that
\[
\cut(G)\subseteq \cut(P).
\]

Suppose now that $J_G$ is Cohen--Macaulay. When $G$ is traceable,  the canonical module $\omega_R$ of $R=S/J_G$ can be determined as follows.

\begin{Proposition}
\label{canonical}
Let $G$ be a traceable graph and $P$ be a spanning  path of $G$. Suppose that  $J_G$ is Cohen--Macaulay  and let  $R=S/J_G$.  Then
\[
\omega_R\iso(\Sect_{T\in\cut(P)\setminus \cut(G)}P_T)/J_P.
\]
\end{Proposition}

\begin{proof}
Since $J_P$ is a complete intersection and a radical ideal with $\height J_P=\height J_G$, we have
\[
\omega_R\iso (J_P:J_G)/J_P.
\]

Thus, by using (\ref{cutsect}) and Lemma~\ref{colon},
\[
\Sect_{T\in\cut(P)}P_T:\Sect_{T\in\cut(G)}P_T=  \Sect_{T\in\cut(P)\setminus \cut(G)}P_T.
\]
\end{proof}




\begin{Example}
{\em
Let $G$ be the complete graph on the vertex set $[n]$ and $P\subset G$ the path $1,2,\ldots, n$. Then $\cut(G) =\emptyset$ and ${i}\subset  [n]$ is a cut set of $P$ if and only $2\leq i\leq n-2$. Hence $\omega_R\iso \Sect_{i=2}^{n-1}(\bar{x}_i,\bar{y}_i)$.
Note that
\[
\Sect_{i=2}^{n-1}(x_i,y_i)=\prod_{i=2}^{n-1}(x_i,y_i)= (\prod_{i\in A}x_i\prod_{i\in [2,n-1]\setminus A}y_i\: \; {A\subseteq [2,n-1]}).
\]
Since $\bar{x}_i\bar{y}_j=\bar{x}_j\bar{y}_i$ for all $j>i$, we finally obtain
\[
\omega_R=(\bar{x}_2\cdots \bar{x}_{n-1}, \ldots,  \bar{x}_2\cdots \bar{x}_{n-2}\bar{y}_{n-1}, \ldots, \bar{x}_2\cdots \bar{x}_j \bar{y}_{j+1}\cdots \bar{y}_{n-1}, \ldots, \bar{y}_2\cdots \bar{y}_{n-1}).
\]
}
\end{Example}

\section{Hibi ideals of K\"onig type}\label{four}

We refer the reader to \cite[pp.~156--159]{HH} for fundamental materials on finite partially ordered sets and finite lattices.  A partially ordered set will be called a poset.

Let $P = \{x_1, \ldots, x_n\}$ be a finite poset with $|P| = n$.  The {\em rank} of each $x \in P$, written as $\rank(x)$, is the maximal length of chains of $P$ of the form
\[
x = x_{i_0} > x_{i_1} > \cdots > x_{i_a}.
\]
Let $\rank(P)$ denote the rank of $P$, which is the maximal length of chains of $P$.  Let $d = \rank(P)$ and write $L_i$ for the set of those $x \in P$ with $\rank(x) = i$, where $0 \leq i \leq d$.

The {\em incomparability graph} of $P$ is the finite simple graph ${\rm incom}(P)$ on $[n] = \{1, \ldots, n\}$ whose edges are those $\{i, j\}$ for which $p_i$ and $p_j$ are incomparable in $P$.  For example,

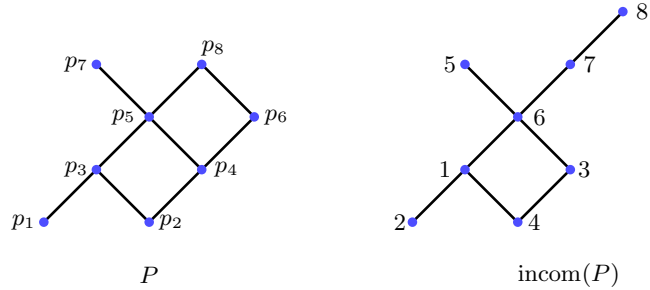
\begin{figure}[hbt]
\begin{center}
\psset{unit=0.7cm}
\begin{pspicture*}(10.9986,3.03)(23.9386,9.5)
\psline[linewidth=1.pt](13.,6.)(12.,5.)
\psline[linewidth=1.pt](14.,5.)(13.,6.)
\psline[linewidth=1.pt](14.,5.)(15.,6.)
\psline[linewidth=1.pt](14.,7.)(13.,6.)
\psline[linewidth=1.pt](15.,6.)(14.,7.)
\psline[linewidth=1.pt](15.,6.)(16.,7.)
\psline[linewidth=1.pt](14.,7.)(15.,8.)
\psline[linewidth=1.pt](16.,7.)(15.,8.)
\psline[linewidth=1.pt](14.,7.)(13.,8.)
\psline[linewidth=1.pt](21.,5.)(22.,6.)
\psline[linewidth=1.pt](22.,6.)(21.,7.)
\psline[linewidth=1.pt](21.,5.)(20.,6.)
\psline[linewidth=1.pt](21.,7.)(20.,6.)
\psline[linewidth=1.pt](22.,8.)(21.,7.)
\psline[linewidth=1.pt](23.,9.)(22.,8.)
\psline[linewidth=1.pt](21.,7.)(20.,8.)
\psline[linewidth=1.pt](20.,6.)(19.,5.)
\begin{scriptsize}
\psdots[dotstyle=*,linecolor=ududff](14.,5.)
\rput[ll] (14.2, 5) {$p_2$}
\psdots[dotstyle=*,linecolor=ududff](13.,6.)
\rput[ll] (12.4, 6) {$p_3$}
\psdots[dotstyle=*,linecolor=ududff](15.,6.)
\rput[ll] (15.25, 6) {$p_4$}
\psdots[dotstyle=*,linecolor=ududff](14.,7.)
\rput[ll] (13.3, 7) {$p_5$}
\psdots[dotstyle=*,linecolor=ududff](16.,7.)
\rput[ll] (16.2, 7) {$p_6$}
\psdots[dotstyle=*,linecolor=ududff](15.,8.)
\rput[ll] (15, 8.3) {$p_8$}
\psdots[dotstyle=*,linecolor=ududff](13.,8.)
\rput[ll] (12.4, 8) {$p_7$}
\psdots[dotstyle=*,linecolor=ududff](12.,5.)
\rput[ll] (11.4, 5) {$p_1$}
\psdots[dotstyle=*,linecolor=ududff](21.,5.)
\rput[ll] (21.2, 5) {$4$}
\psdots[dotstyle=*,linecolor=ududff](20.,6.)
\rput[ll] (19.5, 6) {$1$}
\psdots[dotstyle=*,linecolor=ududff](22.,6.)
\rput[ll] (22.15, 6) {$3$}
\psdots[dotstyle=*,linecolor=ududff](21.,7.)
\rput[ll] (21.3, 7) {$6$}
\psdots[dotstyle=*,linecolor=ududff](22.,8.)
\rput[ll] (22.25, 8) {$7$}
\psdots[dotstyle=*,linecolor=ududff](23.,9.)
\rput[ll] (23.25, 9) {$8$}
\psdots[dotstyle=*,linecolor=ududff](20.,8.)
\rput[ll] (19.6, 8) {$5$}
\psdots[dotstyle=*,linecolor=ududff](19.,5.)
\rput[ll] (18.65, 5) {$2$}
\rput[14] (14, 4) {$P$}
\rput[l4] (21, 4) {${\rm incom}(P)$}
\end{scriptsize}
\end{pspicture*}

\end{center}
\caption{A finite poset and its incomparability graph}
\label{pininside}
\end{figure}


\begin{Lemma}
\label{vertexcovernumber}
Let $P$ be a finite poset of $\rank(P) = d$ and $|P| = n$.  Then
\[
\tau({\rm incom}(P)) = n - (d + 1).
\]
\end{Lemma}

\begin{proof}
Since the induced subgraph of ${\rm incom}(P)$ on $L_i$ is the complete graph on $L_i$, each vertex cover $C$ of ${\rm incom}(P)$ satisfies $|C \cap L_i| \geq |L_i| - 1$.  It then follows that
\[
|C| \geq \sum_{i=0}^{d} (|L_i| - 1) = n - (d + 1).
\]
On the other hand, if
\[
x_{i_0} < x_{i_1} < \cdots < x_{i_d}
\]
is a maximal chain of $P$ of length $d$, then $$L \setminus \{x_{i_0}, x_{i_1}, \ldots, x_{i_d}\}$$ is a vertex cover of ${\rm incom}(P)$.  Hence $\tau({\rm incom}(P)) = n - (d + 1)$, as desired.
\end{proof}

A subset $A$ of a finite poset $P$ is called an {\em antichain} of $P$ if every two distinct elements of $A$ are incomparable in $P$.  A finite poset $P$ is called {\em thin} if $P$ possesses no antichain $A$ with $|A| \geq 3$.  We say that a finite poset $P$ of $\rank(P) = d$ is {\em pure} if every maximal chain of $P$ is of length $d$.

\begin{Lemma}
\label{dilworth}
Let $P$ be a finite pure poset.  Then the following conditions are equivalent:
\begin{itemize}
\item[(i)]
$P$ is thin.
\item[(ii)]
each $|L_i|$ satisfies $|L_i| \leq 2$.
\item[(iii)]
$P$ is the sum of at most two disjoint chains.
\item[(iv)]
${\rm incom}(P)$ is bipartite.
\end{itemize}
\end{Lemma}

\begin{proof}
The equivalence (i) $\Leftrightarrow$ (iii) is just a decomposition theorem of Dilworth \cite{Dilworth}.  Furthermore, (iii) $\Rightarrow$ (ii) is clear.  On the other hand, in order to show (ii) $\Rightarrow$ (i), suppose that each $|L_i| \leq 2$ and that $\{ a, b, c \}$ is an antichain of $P$.  Since $P$ is pure, if $\rank(a) = \rank(b)$ and, say, $\rank(c) < \rank(a)$, then either $c < a$ or $c < b$, a contradiction.  If $\rank(a) < \rank(b) < \rank(c)$, then there is $b' \in P$ with $\rank(b') = \rank(b)$ and $a < b' < c$.  Hence $a < c$, a contradiction.  As a result, no three-element subset of $P$ can be an antichain of $P$.  Finally, the equivalence (iii) $\Leftrightarrow$ (iv) follows directly from the definition of ${\rm incom}(P)$.
\end{proof}

The finite nonpure poset $P$ of $\rank(P) = d$ drawn below satisfies each $|L_i| \leq 2$.  However, there is an antichain $A$ of $P$ with $|A| = d + 1$.

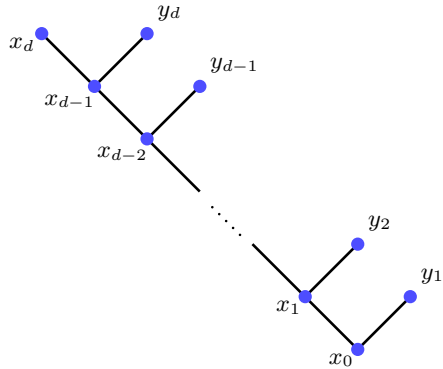
\begin{figure}[hbt]
\begin{center}
\newrgbcolor{ududff}{0.30196078431372547 0.30196078431372547 1.}
\psset{xunit=0.7cm,yunit=0.7cm,algebraic=true,dimen=middle,dotstyle=o,dotsize=5pt 0,linewidth=2.pt,arrowsize=3pt 2,arrowinset=0.25}
\begin{pspicture*}(12.86,2.5)(22.4,11.76)
\psline[linewidth=1.pt](16.,8.)(15.,9.)
\psline[linewidth=1.pt](19.,5.)(20.,6.)
\psline[linewidth=1.pt](16.,8.)(17.,9.)
\psline[linewidth=1.pt](15.,9.)(16.,10.)
\psline[linewidth=1.pt](15.,9.)(14.,10.)
\psline[linewidth=1.pt](21.,5.)(20.,4.)
\psline[linewidth=1.pt](19.,5.)(20.,4.)
\psline[linewidth=1.pt](16.,8.)(17.,7.)
\psline[linewidth=1.pt](19.,5.)(18.,6.)
\psline[linewidth=1.pt,linestyle=dotted](17.22,6.72)(17.78,6.16)
\begin{scriptsize}
\psdots[dotstyle=*,linecolor=ududff](19.,5.)
\rput[ll] (18.45, 4.8) {$x_1$}
\psdots[dotstyle=*,linecolor=ududff](16.,8.)
\rput[ll] (15.05, 7.7) {$x_{d-2}$}
\psdots[dotstyle=*,linecolor=ududff](15.,9.)
\rput[ll] (14.05, 8.7) {$x_{d-1}$}
\psdots[dotstyle=*,linecolor=ududff](21.,5.)
\rput[ll] (21.2, 5.4) {$y_1$}
\psdots[dotstyle=*,linecolor=ududff](20.,6.)
\rput[ll] (20.2, 6.4) {$y_2$}
\psdots[dotstyle=*,linecolor=ududff](17.,9.)
\rput[ll] (17.2, 9.4) {$y_{d-1}$}
\psdots[dotstyle=*,linecolor=ududff](16.,10.)
\rput[ll] (16.2, 10.4) {$y_d$}
\psdots[dotstyle=*,linecolor=ududff](14.,10.)
\rput[ll] (13.38, 9.8) {$x_d$}
\psdots[dotstyle=*,linecolor=ududff](20.,4.)
\rput[ll] (19.45, 3.8) {$x_0$}
\end{scriptsize}
\end{pspicture*}

\end{center}
\caption{A nonpure poset and an antichain}
\label{pininside}
\end{figure}


We now come to a characterization for ${\rm incom}(P)$ of a finite pure poset to be K\"onig.

\begin{Theorem}
\label{incimparability}
The incomparability graph ${\rm incom}(P)$ of a finite pure poset $P$ is K\"onig if and only if ${\rm incom}(P)$ is bipartite.
\end{Theorem}

\begin{proof}
Every bipartite graph is K\"onig.  Now, suppose that ${\rm incom}(P)$ is K\"onig.  Let $q$ denote the number of those $L_i$ with $|L_i| = 1$.  Since $P$ is pure, if $|L_i| = 1$ and $\rank(a) = i$, then $a$ is comparable to every element of $P$.  Thus
\[
n - q = \sum_{|L_i| > 1} |L_i| \geq 2 \cdot m({\rm incom}(P)) = 2 \cdot \tau({\rm incom}(L)) = 2(n - (d + 1)).
\]
Hence
\[
\sum_{|L_i| > 1} |L_i| \geq 2 \big(\big(\sum_{|L_i| > 1} |L_i|) + q\big) - (d + 1)\big).
\]
It then follows that
\[
\sum_{|L_i| > 1} |L_i| \leq 2(d - q + 1).
\]
Since the number of those $0 \leq i \leq d$ with $|L_i| > 1$ is $d - q + 1$, it follows that
\[
\sum_{|L_i| > 1} |L_i| \geq 2(d - q + 1).
\]
Hence
\[
\sum_{|L_i| > 1} |L_i| = 2(d - q + 1).
\]
Thus $|L_i| = 2$ if $|L_i| > 1$.  It then follows from Lemma \ref{dilworth} that ${\rm incom}(P)$ is bipartite, as desired.
\end{proof}

We now turn to the study of the Hibi ring \cite{Hibi} of a finite distributive lattice.  Let $P = \{p_1, \ldots, p_d\}$ be a finite poset.  A subset $\alpha \subset P$ is called a {\em poset ideal} of $P$ if $\alpha$ enjoys the property that if $p_i \in \alpha, p_j \in P$ together with $p_j \leq p_i$, then $p_j \in \alpha$.  Let $\Jc(P)$ denote the set of poset ideals of $P$.  It then follows that $\Jc(P)$ is a finite poset, in fact, finite lattice, ordered by inclusion.  Furthermore, the finite lattice $\Jc(P)$ is distributive.  On the other hand, Birkhoff's fundamental structure theorem \cite[Theorem 9.1.7]{HH} guarantees that, given a finite distributive lattice $L$, there is a unique finite poset $P$ with $L = \Jc(P)$.  Every finite distributive lattice is pure.

Let $T = K[z_1, \ldots, z_d, t]$ denote the polynomial ring in $d + 1$ variables over a field $K$.  Given a poset ideal $\alpha$ of $P$, we associate the squarefree monomial
\[
u_\alpha = (\prod_{p_i \in \alpha} z_i) t \in T.
\]
In particular $u_\emptyset = t$ and $u_P = z_1 \cdots z_d t$.  In \cite{Hibi}, the toric ring $\Rc_K[\Jc(P)]$ which is generated by those monomials $u_\alpha$ with $\alpha \in \Jc(P)$ is introduced:
\[
\Rc_K[\Jc(P)] = K[\{u_\alpha\}_{\alpha \in \Jc(P)}].
\]
Let $S = K[\{x_\alpha\}_{\alpha \in \Jc(P)}]$ denote the polynomial ring in $|\Jc(P)|$ variables over $K$.  We introduce a surjective ring homomorphism $\pi : S \to \Rc_K[\Jc(P)]$ by setting
\[
\pi(x_\alpha) = u_{\alpha}.
\]
It turns out that the toric ideal $I_{\Jc(P)} \subset S$ of $\Rc_K[\Jc(P)]$ which is the kernel of $\pi$ is
\[
I_{\Jc(P)} = (x_{\alpha_i}x_{\alpha_j} - x_{\alpha_i \cap \alpha_j}x_{\alpha_i \cup \alpha_j} \, : \, 1 \leq i < j \leq n).
\]

Nowadays, the toric ring $\Rc_K[L]$ with $L = \Jc(P)$ is called the {\em Hibi ring} of the finite distributive lattice $L$.  One can ask when the toric ideal $I_L$ of $\Rc_K[L]$ is of K\"onig type.

Let $L = \Jc(P) = \{x_1, \ldots, x_n\}$ with $n = |\Jc(P)|$ and $S = K[x_1, \ldots, x_n]$.  We discuss a reverse lexicographic order $<_{\rm rev}$ on $S$ with the property that $x_i <_{\rm rev} x_j$ if $x_i < x_j$ in $L$.  A basic fact on $<_{\rm rev}$ is

\begin{Lemma}[\cite{Hibi}]
\label{revlex}
The set of those binomials $$f_{i, j} = x_ix_j - (x_i \wedge x_j)(x_i \vee x_j)$$ for which $x_i$ and $x_j$ are incomparable in $L$ is the reduced Gr\"obner basis of $I_L$ with respect to $<_{\rm rev}$.
\end{Lemma}

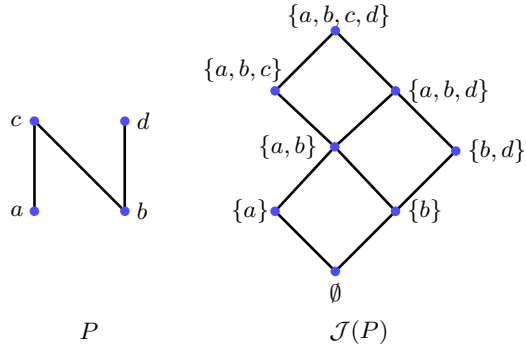
\begin{figure}[hbt]
\label{simpleexample}
\begin{center}
\psset{unit=0.4cm}
\begin{pspicture*}(2.91763636363636,-1)(22.73636363636358,11.884545454545451)
\psline[linewidth=1.pt](6.,4.)(6.,7.)
\psline[linewidth=1.pt](9.,4.)(9.,7.)
\psline[linewidth=1.pt](9.,4.)(6.,7.)
\psline[linewidth=1.pt](16.,2.)(18.,4.)
\psline[linewidth=1.pt](14.,4.)(15.98,6.14)
\psline[linewidth=1.pt](18.,4.)(15.98,6.14)
\psline[linewidth=1.pt](16.,2.)(14.,4.)
\psline[linewidth=1.pt](18.,4.)(20.,6.)
\psline[linewidth=1.pt](15.98,6.14)(18.,8.)
\psline[linewidth=1.pt](20.,6.)(18.,8.)
\psline[linewidth=1.pt](15.98,6.14)(14.,8.)
\psline[linewidth=1.pt](18.,8.)(16.,10.)
\psline[linewidth=1.pt](14.,8.)(16.,10.)
\begin{scriptsize}
\psdots[dotstyle=*,linecolor=ududff](6.,4.)
\rput[ll] (5.2, 4) {$a$}
\psdots[dotstyle=*,linecolor=ududff](6.,7.)
\rput[ll] (5.2, 7) {$c$}
\psdots[dotstyle=*,linecolor=ududff](9.,4.)
\rput[ll] (9.4, 4) {$b$}
\psdots[dotstyle=*,linecolor=ududff](9.,7.)
\rput[ll] (9.4, 7) {$d$}
\psdots[dotstyle=*,linecolor=ududff](16.,2.)
\rput[ll] (15.8, 1.3) {$\emptyset$}
\psdots[dotstyle=*,linecolor=ududff](18.,4.)
\rput[ll] (18.4, 4) {$\{b\}$}
\psdots[dotstyle=*,linecolor=ududff](14.,4.)
\rput[ll] (12.55, 4) {$\{a\}$}
\psdots[dotstyle=*,linecolor=ududff](15.98,6.14)
\rput[ll] (13.45, 6.14) {$\{a,b\}$}
\psdots[dotstyle=*,linecolor=ududff](20.,6.)
\rput[ll] (20.4, 6) {$\{b,d\}$}
\psdots[dotstyle=*,linecolor=ududff](18.,8.)
\rput[ll] (18.4, 8) {$\{a,b,d\}$}
\psdots[dotstyle=*,linecolor=ududff](14.,8.)
\rput[ll] (11.6, 8.7) {$\{a,b,c\}$}
\psdots[dotstyle=*,linecolor=ududff](16.,10.)
\rput[ll] (14.4, 10.5) {$\{a,b,c,d\}$}
\rput[ll] (7.5, 0) {$P$}
\rput[ll] (15.8, 0) {$\mathcal{J}(P)$}
\end{scriptsize}
\end{pspicture*}

\end{center}
\caption{A poset and its poset ideal lattice}
\label{pininside}
\label{simpleexample}
\end{figure}

We now come to a characterization for the toric ideal $I_L$ to be of K\"onig type.

\begin{Theorem}
\label{binomialhibi}
Let $L$ be a finite distributive lattice and $I_L$ the toric ideal of the Hibi ring $\Rc_K[L]$ of $L$.  Then the following conditions are equivalent:
\begin{itemize}
\item[(i)]
$I_L$ is of K\"onig type with respect to $f_{i, j}$ and $<_{\rm rev}$.
\item[(ii)]
$L$ is thin.
\item[(iii)]
${\rm incom(L)}$ is bipartite.
\end{itemize}
\end{Theorem}

\begin{proof}
Let $L = \Jc(P)$ with $|P| = d$ and $|L| = n$.  Since $$\dim \Rc_K[L] = \rank(L) + 1 = d + 1$$ and $\Rc_K[L]$ is Cohen--Macaulay, one has $$\height I_L = n - (d + 1).$$  Let, as before, $L_i$ denote the set of those $\alpha \in L = \Jc(P)$ with $\rank(\alpha) = i$, where $0 \leq i \leq d$.  Let $q \geq 2$ denote the number of those $0 \leq i \leq d$ with $|L_i| = 1$.

(ii) $\Rightarrow$ (i):
Suppose that $L$ is thin.  Lemma \ref{dilworth} says that $|L_i| \leq 2$ for each $0 \leq i \leq d$.   One has $n = q + 2(d - q + 1) = 2d - q + 2$.  Hence $\height I_L = n - (d + 1) = d - q + 1$.  Furthermore, the number of those $i$ with $|L_i| = 2$ is $d - q + 1$.  If $|L_i| = 2$ and $L_i = \{x_{r_i}, x_{r'_i}\}$, then the binomial $f_{r_i,r'_i}$ belongs to the reduced Gr\"obner basis of $I_L$ with respect to $<_{\rm rev}$ and $\ini_{<_{\rm rev}}(f_{r_i,r'_i}) = x_{r_i}x_{r'_i}$.  Clearly, the $d - q + 1$ monomials $x_{r_i}x_{r'_i}$ with $|L_i| = 2$ form a regular sequence.  Hence $I_L$ is of K\"onig type, as desired.

(i) $\Rightarrow$ (iii):
Suppose that $I_L$ is of K\"onig type.  Then there are $n - (d + 1)$ binomials $f_{i, j}$ for which their initial monomials $x_ix_j$ form a regular sequence.  Hence the matching number of ${\rm incom(L)}$ is at least $n - (d + 1)$.  Since Lemma \ref{vertexcovernumber} says that the vertex cover number of ${\rm incom(L)}$ is $n - (d + 1)$, it follows that ${\rm incom(L)}$ is K\"onig.  Now, Theorem \ref{incimparability} guarantees that ${\rm incom(L)}$ is bipartite.

(iii) $\Leftrightarrow$ (ii):
This follows from Lemma \ref{dilworth}.
\end{proof}

\begin{Corollary}
\label{segre}
Let $S = K[x_1, \ldots, x_n]$ and $T = K[y_1, \ldots, y_m]$ denote polynomial rings in $n \geq 2$ and in $m \geq 2$ variables over a field $K$.  Then the defining ideal of the Segre product $S\,\sharp\,T$ is of K\"onig type if and only if either $n = 2$ or $m = 2$.
\end{Corollary}

\begin{proof}
Let $P$ be a finite poset which is the disjoint union of the chain of length $n - 2$ and the chain of length $m - 2$.  Then the Segre product $S\,\sharp\,T$ coincides with $\Rc_K[\Jc(P)]$.  The finite distributive lattice $L = \Jc(P)$ is thin if and only if either $n = 2$ or $m = 2$.  Now, Theorem \ref{binomialhibi} guarantees that the defining ideal of $S\,\sharp\,T$ is of K\"onig type if and only if either $n = 2$ or $m = 2$, as desired.
\end{proof}

By using Lemma~\ref{colon} can  give an explicit description of the  canonical module of $\Rc_K[L]$ when $L$ is thin, in the same way as we did it for edge  ideals and binomial edge ideals. More generally, let $L$  be a distributive lattice. We call two   incomparable elements  $\alpha,\beta\in L$ a {\em basic pair} of $L$,  if $\alpha$ and $\beta$ cover $\alpha\wedge \beta$ and $\alpha\vee \beta$ covers $\alpha$ and $\beta$. If $L$ is thin, the binomials $f_{i,j}$ in Theorem~\ref{binomialhibi} are just the binomials $f_{\alpha,\beta}=x_{\alpha}x_{\beta} - x_{\alpha \wedge  \beta} x_{\alpha \vee  \beta}$ where  $\alpha$ and $\beta$ form  a basic pair in $L$. The elements  $\alpha, \beta, \alpha\wedge \beta, \alpha\vee \beta$ are  the corners of a square, also called a cell, whose  four edges are $\{\alpha,\alpha\wedge \beta\}$, $\{\alpha,\alpha\vee \beta\}$, $\{\beta,\alpha\wedge \beta\}$ and $\{\beta,\alpha\vee \beta\}$.  For example,  the lattice in Figure~\ref{simpleexample} has three cells. One of the cells has the corners $\{a,b\}$, $\{b\}$, $\{b,d\}$ and $\{a, b,d\}$. The set of cells of $L$ is  a configuration of cells $\Cc$, as defined in \cite{HH2}. We denote by  $I_\Cc$ the ideal generated by the binomials $f_{\alpha,\beta}$ for which $\alpha$ and $\beta$ is  a basic pair. By Theorem~\ref{binomialhibi}, the generators of $I_\Cc$ form a regular sequence and  $\height I_\Cc=\height I_L$. Since $I_L$ is Cohen--Macaulay, it follows that $(I_\Cc:I_L )/I_\Cc$ is the canonical module $\Rc_K[L]$.

In order compute $I_\Cc:I_L$ we use the fact that $I_\Cc$ is a radical ideal because  its initial ideal is squarefree, and we use the presentation
\[
I_\Cc=\Sect_WP_W(\Cc),
\]
where intersection is taken over all admissible sets $W$  of $\Cc$, see  \cite[Theorem 3.2]{HH2}. Recall that $W$ is called {\em admissible}, if whenever $W$ intersected with the corners of a cell of $\Cc$ is nonempty, then $W$ contains an edge of this cell. The  prime ideal   $P_W(\Cc)$ is the ideal generated by the variables corresponding to $W$ and the inner minors of $\Cc'$, where $\Cc'$ is the set of cells of $\Cc$ which do not intersect $W$. Note that the union of the cells of $\Cc'$ is a disjoint union of distributive lattices, say $L_1,\ldots, L_m$. Then the ideal $J_W$ generated by the inner minors of $\Cc'$ is just $\sum_{i=1}^m I_{L_i}$.

By applying Lemma~\ref{colon} we find that
\[
I_\Cc:I_L=\Sect_{W\atop W\neq \emptyset}(W,J_W), \quad \text{$W$ is admissible and $\height (W,J_W)=\height I_L$}.
\]
We come back to the example given in Figure~\ref{simpleexample}. In order to simplify notation we relabel the vertices of $L$ as follows: $1=\emptyset$, $2= \{a\}$, $3=\{b\}$, $4=\{a,b\}$, $5=\{b,d\}$, $6=\{a,b,c\}$, $7=\{a,b,d\}$ and $8=\{a,b,c,d\}$.

The admissible sets $W$ are $\{1,2\}$, $\{6,8\}$, $\{1,5,3\}$,$\{2,4,6\}$, $\{5,7,8\}$,  $\{3,4,6\}$ and $\{3,4,7\}$. Then for example
$P_W(\Cc)=(x_1,x_2, x_4x_5-x_3x_7, x_6x_7-x_4x_8,x_6x_5-x_3x_8)$ for $W=\{1,2\}$, while for $W=\{3,4,7,\}$ we have $P_W(\Cc)=(x_3,x_4,x_7)$. The only  ideals of height $3$ (which is the height of $I_\Cc$)   among the prime ideals $P_W(\Cc)$ are the ideals $(x_3,x_4,x_7)$, $(x_2,x_4,x_7)$ and $(x_3,x_4,x_6)$. Therefore.
\[
I_\Cc:I_L=(x_3,x_4,x_7)\sect (x_2,x_4,x_7)\sect (x_3,x_4,x_6)=(x_4, x_2x_3, x_3x_7, x_6x_7).
\]
In \cite[Theorem 3.3]{HH2} the minimal prime ideals among the $\P_W(\Cc)$ are  described,  in general.

\medskip
In Theorem~\ref{binomialhibi} we characterized for which finite distributive lattices $L$ the toric ideal $I_L$ of the Hibi ring $\Rc_K[L]$ is of K\"onig type with respect to the reverse lexicographic monomial order. One may ask the same question for any other monomial order. We do not  have a complete answer to this question.  But it is quite obvious that $I_L$ being of K\"onig type is rather restrictive, no matter which monomial order we choose. Indeed, suppose that $I_L$ is of K\"onig type for the monomial order $<$. Then $\ini_<(I_L)$ contains a regular sequence $u_1,\ldots,u_h$ of monomials of degree $2$, where $h=\height I_L$. Let $T$ be the set of join irreducible elements of $L$. Then $h=|L|-|T|-1$. It follows that if $I_L$ is of K\"onig type, then $2(|L|-|T|-1)\leq |L|$,  equivalently, $|L|\leq 2(|T|+1)$.

Consider for example the Boolean lattice $B_n$ of rank $n$. Then $|B_n|=2^n$ and $|T|=n$ and the inequality $2^n\leq 2(n+1)$ is satisfied only if $n\leq 3$. So for $n\geq 4$, $I_{B_n}$ is never of K\"onig type.

On the other hand,  $I_{B_2}$ is a principle ideal, and so is of  K\"onig type, and $I_{B_3}$ is of K\"onig type with respect to the lexicographic monomial order. Indeed, let us  label $B_3$  as indicated in Figure~\ref{Boolean}.

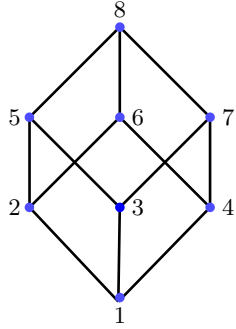
\begin{figure}[hbt]
\label{Boolean}
\begin{center}
\psset{unit=0.4cm}
\begin{pspicture*}(5.947228381374732,1.9403991130820422)(25.051441241685165,14.995831485587589)
\psline[linewidth=1.pt](16.,14.)(13.,11.)
\psline[linewidth=1.pt](16.,14.)(19.,11.)
\psline[linewidth=1.pt](13.,11.)(16.,8.)
\psline[linewidth=1.pt](19.,11.)(16.,8.)
\psline[linewidth=1.pt](16.,11.)(13.,8.)
\psline[linewidth=1.pt](16.,11.)(19.,8.)
\psline[linewidth=1.pt](13.,8.)(15.951662971175182,4.8672283813747255)
\psline[linewidth=1.pt](19.,8.)(15.951662971175182,4.8672283813747255)
\psline[linewidth=1.pt](16.,14.)(16.,11.)
\psline[linewidth=1.pt](13.,11.)(13.,8.)
\psline[linewidth=1.pt](19.,11.)(19.,8.)
\psline[linewidth=1.pt](16.,8.)(15.951662971175182,4.8672283813747255)
\begin{scriptsize}
\psdots[dotstyle=*,linecolor=ududff](16.,14.)
\rput[ll] (15.8, 14.6) {$8$}
\psdots[dotstyle=*,linecolor=ududff](13.,11.)
\rput[ll] (12.3, 11) {$5$}
\psdots[dotstyle=*,linecolor=ududff](19.,11.)
\rput[ll] (19.4, 11) {$7$}
\psdots[dotstyle=*,linecolor=ududff](16.,11.)
\rput[ll] (16.4, 11) {$6$}
\psdots[dotstyle=*,linecolor=ududff](13.,8.)
\rput[ll] (12.3, 8) {$2$}
\psdots[dotstyle=*,linecolor=ududff](19.,8.)
\rput[ll] (19.4, 8) {$4$}
\psdots[dotstyle=*,linecolor=ududff](16,5)
\rput[ll] (15.8, 4.4) {$1$}
\psdots[dotstyle=*,linecolor=blue](16.,8.)
\rput[ll] (16.4, 8) {$3$}
\end{scriptsize}
\end{pspicture*}
\end{center}
\caption{A Boolean lattice $B_3$}
\label{pininside}
\end{figure}

Then
\begin{eqnarray*}
I_{B_3}&=& (x_2x_3-x_1x_5,x_3x_4-x_1x_7,x_1x_6-x_2x_4,x_5x_6-x_2x_8,x_6x_7-x_4x_8,\\
&& x_2x_7-x_1x_8,x_4x_5-x_1x_8,x_5x_7-x_3x_8, x_3x_6-x_1x_8).
\end{eqnarray*}
With respect to the lexicographic monomial order  induced by $x_1>x_2>\cdots >x_8$ we find that
\[
\ini_<(I_{B_3})=
(x_1x_8, x_3x_8, x_3x_6, x_2x_7, x_4x_8, x_2x_8, x_1x_6, x_1x_7, x_1x_5).
\]
This initial ideal contains the regular sequence
\[
x_3x_6,   x_2x_7,  x_4x_8,  x_1x_5.
\]

The corresponding ideal whose initial monomials  are this sequence of monomials  is the ideal
\[
J=(x_3x_6-x_4x_5,x_2x_7-x_3x_6,x_4x_8-x_6x_7,x_1x_5-x_2x_3).
\]
Since $\height I_{B_3}=4$, this shows  $I_{B_3}$ is of König type with respect to the lexicographic monomial order.

\end{document}